\documentclass[12pt, reqno]{amsart}
\usepackage{amsmath, amsthm, amscd, amsfonts, amssymb, graphicx, color, mathrsfs}
\usepackage[bookmarksnumbered, colorlinks, plainpages]{hyperref}
\usepackage{pgfplots}
\usepackage{slashed}

\newtheorem{theorem}{Theorem}[section]
\newtheorem{lemma}[theorem]{Lemma}

\newtheorem{proposition}[theorem]{Proposition}
\newtheorem{corollary}[theorem]{Corollary}

\theoremstyle{definition}
\newtheorem{definition}[theorem]{Definition}

\theoremstyle{remark}
\newtheorem{remark}[theorem]{Remark}
\numberwithin{equation}{section}

\newcommand*\diff{\mathrm{d}}
\DeclareMathOperator{\supp}{supp}

\begin{document}
\setcounter{page}{1}

\title[Estimates for $S(m,g)$ calculus]{Kernel estimates for the Weyl-H\"ormander calculus, Weak (1,1) boundedness and other regularity properties}

\author[D. Cardona Sanchez]{Duv\'an Cardona}
\address{
  Duv\'an Cardona:
 \endgraf
  King Fahd University of Petroleum and Minerals 
   \endgraf
    College of Computing and Mathematics
    \endgraf
    Dhahran,  Saudi Arabia.
  \endgraf
  Pontificia Universidad Javeriana, 
  \endgraf
  Department of Mathematics
  \endgraf
  Bogot\'a-Colombia.
  \endgraf
  {\it E-mail address:} {\rm duvan.sanchez@kfupm.edu.sa; duvanc306@gmail.com}  
  }

 \author[J. Delgado]{Julio Delgado}
\address{
  Julio Delgado:
  \endgraf
  Departmento de Matematicas
  \endgraf
  Universidad del Valle
  \endgraf
  Cali-Colombia
  \endgraf
    {\it E-mail address} {\rm delgado.julio@correounivalle.edu.co}}

\author[M. Mart\'inez-Flores]{Manuel Mart\'inez-Flores}
\address{
  Manuel Mart\'inez-Flores
  \endgraf
  Department of Mathematics
  \endgraf
 Universidad del Valle de Guatemala, Guatemala
  \endgraf
  {\it E-mail address} {\rm mar21403@uvg.edu.gt, manuelalejandromartinezf@gmail.com}
  }

\thanks{ Manuel Mart\'inez Flores has been supported by the {\it{Liderazgo en Ciencias}} scholarship of the Universidad del Valle de Guatemala.  Duv\'an Cardona has been supported by the Department of Mathematics at KFUPM (King Fahd University of Petroleum and Minerals),  by the FWO  Odysseus  1  grant  G.0H94.18N:  Analysis  and  Partial Differential Equations, by the Methusalem programme of the Ghent University Special Research Fund (BOF)
(Grant number 01M01021) and has been supported by the FWO Fellowship
Grant No 1204824N of the Belgian Research Foundation FWO. He also has been supported by the Oberwolfach Leibniz Fellow of the Mathematical Institute of Research of Oberwolfach, MFO-Germany, Project F2511 (2026) and by the Department of Mathematics of the Pontificia Universidad Javeriana, Bogot\'a-Colombia. Julio Delgado has been supported by the Department of Mathematics at Universidad del Valle, Cali-Colombia, and by the Global Minds Grant: PDE and Control Theory of Ghent University, Belgium}


\begin{abstract}
   We provide kernel estimates for pseudo-differential operators in the $S(m,g)$ Weyl-H\"ormander calculus  classes. These new results lead to boundedness theorems like the weak (1,1) type for these classes and the boundedness of these operators from Hardy spaces to Lebesgue spaces $L^p$. We prove that under a suitable hypoelliticity condition on the weigh $m,$  one can get the weak (1,1) inequality for the Fefferman order established previously by the second author in analogy to the classical $L^p$-boundedness theorem due to Fefferman. Our approach makes an extension to $S(m,g)$ classes of previous estimates proved by \'Alvarez and Hounie, Nagase, Kumano-Go, etc, for the H\"ormander classes including new regularity properties in this framework from the Hardy space $H^p$ to $L^p$ and also on Lorentz spaces. We have identified that the {\it Uniform H\"ormander Condition} plays a fundamental role in some of our boundedness theorems.
\end{abstract} 

\maketitle
\tableofcontents
\allowdisplaybreaks

\section{Introduction}

\subsection{Overview} From the perspective of harmonic analysis, the boundedness of pseudo-differential operators is a fundamental problem that, depending on the class of symbols under consideration, requires a variety of techniques. For the classical {\it Kohn--Nirenberg classes}, boundedness properties can be reduced to those of {\it Calder\'on--Zygmund singular integral operators} \cite{CalderonZygmund1952}. For more sophisticated symbol classes, particularly those containing {\it oscillating Fourier multipliers} as fundamental examples, such as the $(\rho,\delta)$-{\it H\"ormander classes} \cite{Hormander1985III}, {\it Littlewood--Paley theory} and {\it complex interpolation techniques} became essential tools following the pioneering work of Fefferman \cite{Fefferman1973}.

The aim of this paper is to investigate the mapping properties of the {\it Weyl--H\"ormander calculus} associated to the $S(m,g)$-classes for suitable metrics $g$. These classes have been shown to be a versatile family of symbols generalising fundamental pseudo-differential calculi, indeed, heuristically one has that
\begin{equation*}
{\textnormal{ Kohn-Nirenberg }S^\mu\subset \textnormal{ H\"ormander } S^\mu_{\rho,\delta} \subset \textnormal{ Beals-Fefferman } S^{\Phi,\phi}\subset S(m,g) \textnormal{ classes},}
\end{equation*}
and also of particular interest for operators with discrete spectra 
\begin{equation*}
  {\textnormal{ Shubin classes } \Lambda^{\mu},\quad m=(1+|x|+|\xi|)^\mu \subset S(m,g)= \textnormal{Weyl-H\"ormander classes}.}
\end{equation*}
The Fefferman $L^p$-boundedness theorem for these classes was established by the second  author around 20 years ago, and our goal here is to obtain further estimates, including a weak $(1,1)$-boundedness result, by adapting the approach developed by \'Alvarez and Hounie \cite{alvarez-hounie}. It could be understood as an approach seeking to provide the mapping properties of the operators making use of suitable kernel estimates and having a replacement for the Calder\'on-Vaillancourt theorem, or in other words, the corresponding $L^2$-boundedness theorem. In our setting, a more general result regarding the $L^p$-boundedness for operators in the class $S(1,g)$ was first proved by Richard Beals \cite{Be,Be2}.   Since our results here extend those of Fefferman \cite{Fefferman1973},  we recall the following sharp $L^p$-estimates for the H\"ormander classes combining different versions, particularly from Fefferman \cite{Fefferman1973} and \'Alvarez and Hounie \cite{alvarez-hounie}. For simplicity, we present only the case $\delta\leq \rho.$
\begin{theorem}[$L^p$-bounds for H\"ormander classes]\label{FT}
Let $A:C^\infty(\mathbb{R}^n)\rightarrow\mathscr{D}'(\mathbb{R}^n)$ be a pseudo-differential operator with symbol $\sigma\in S^{-m}_{\rho,\delta}(\mathbb{R}^n\times {\mathbb{R}}^n ),$  $0\leq \delta\leq \rho\leq 1,$ $\delta<1.$ Then,
\begin{itemize}
    \item{\textnormal{(a)}} if $m=\frac{n(1-\rho)}{2},$ then $A$ extends to a bounded operator from $L^\infty(\mathbb{R}^n)$ to $ \mathrm{BMO}(\mathbb{R}^n)$ and also admits a bounded extension from the Hardy space $H^1(\mathbb{R}^n)$ to $L^1(\mathbb{R}^n)$. Moreover, for all $1<p<\infty,$ $A$ admits a bounded extension on $L^p(\mathbb{R}^n).$
   \item{\textnormal{(b)}} If $m\geq m_{p}:= n(1-\rho)\left|\frac{1}{p}-\frac{1}{2}\right|,$ then $A$ extends to a bounded operator on $ L^p(\mathbb{R}^n),$ for $1<p<\infty.$ 
   \item{\textnormal{(c)}} if $m=\frac{n(1-\rho)}{2},$ then $A$ extends to a bounded operator from $L^1(\mathbb{R}^n)$ to $ L^{1,\infty}(\mathbb{R}^n),$ or equivalently, $A$ is of weak (1,1) type.
\end{itemize}
\end{theorem} 
\subsection{Regularity properties of $S(m,g)$ classes}
The pseudo-differential operators considered in this work have symbols in the Weyl-H\"ormander classes $S(m^{-n\varepsilon},g).$ These classes are associated to the metric
$
  g_{(x,\xi)}(\diff x,\diff\xi)=m(x,\xi)^{-2}(\langle \xi \rangle^2 \diff x^2+\diff\xi^2),$ $x,\xi\in\mathbb{R}^n,$ 
and to  the weight
$
  m(x,\xi)=(a(x,\xi)+\langle \xi\rangle)^{\frac{1}{2}}.
$
Here, the symbol $a:=a(x,\xi),$ is the principal part of a second order differential operator $L,$ with coefficients of bounded derivatives and defined by
\begin{equation*}\label{ST}
  {Lf}=-\sum_{i,j}a_{ij}(x)\frac{\partial^2}{\partial{x_i\partial x_j}}f+\sum_{k}b_k(x)\frac{\partial}{\partial x_k}f+c(x)f,\,\,\, f\in C^\infty_{0}(\mathbb{R}^n).
\end{equation*}
We also assume for every $x\in\mathbb{R}^n,$ the matrix $A(x)=(a_{ij}(x))$ is a positive semi-definite matrix  of rank $r(x):=\textnormal{rank}(A(x))\geq r_0\geq 1.$ The coefficients $a_{ij}$ are  smooth functions which are uniformly bounded on $\mathbb{R}^n$, together with all their derivatives. 
Since there is a family of vector fields $X_j$ such that (see Subsection \ref{UHC:Subsec})
\begin{equation}\label{X:UHC}
    A(x) = \sum_{j=1}^k X_j^*(x)X_j(x),
\end{equation}
we will consider that the system $\mathbb{X}=\left\{X_j\right\}$ satisfies the H\"ormander Condition at step $r\in \mathbb{N}.$ 
Important examples  arise from operators of the form
\begin{equation*}\label{Ex1}
  L=-\sum_{j}{\mathcal{X}}_{j}^{*}{\mathcal{X}}_{j}+{\mathcal{X}}_0,\,\,\,L=-\sum_{j}{\mathcal{X}}_{j}^{*}{\mathcal{X}}_{j},
\end{equation*} where $\left\{{\mathcal{X}}_{i}:1\leq i\leq k\right\}$ is a system of H\"ormander vector fields. Such operators were introduced by H\"ormader in his celebrated paper {\em Hypoelliptic second order differential equations} (cf. \cite{Hormander1967a}) which as of today counts 1369 citations in MathSciNet and has opened intensive research in  multiple directions.
\subsubsection{$L^p$-regularity properties for $S(m,g)$-classes}Let us define
\begin{equation*}\label{q0}
  Q_0:=r_0+2(n-r_0),\,\,\varepsilon_0:=\frac{Q_0}{2n}-\frac{1}{2}. 
\end{equation*} It was proved by the second author the following Fefferman-type boundedness theorem. Below, the $\mathrm{BMO}$ space is adapted to a subelliptic distance associated to the weight $m=(a+\langle\xi\rangle)^\frac{1}{2},$ which, in the case where $a$ is an elliptic symbol, recovers the classical $\mathrm{BMO}$ space of John and Nirenberg; the Hardy space $H^1$ is the dual space of $\mathrm{BMO}.$

\begin{theorem}[$L^p$-bounds for  $S(m,g)$ classes \cite{Delgado2006,  Delgado2016}]\label{Th:Delgado:2005} Let $A:C^\infty(\mathbb{R}^n)\rightarrow\mathscr{D}'(\mathbb{R}^n)$ be a pseudo-differential operator with symbol $\sigma\in S(m^{-\beta},g).$ Then, the following statements hold. 
\begin{itemize}
    \item{\textnormal{(a)}} Let $\beta=n\varepsilon$ where  $\varepsilon_0\leq \varepsilon<\frac{Q_0}{2n}.$ Then  $A$ extends to a bounded operator from $L^\infty(\mathbb{R}^n)$ to $ \mathrm{BMO}(\mathbb{R}^n)$ and also admits a bounded extension from the Hardy space $H^1(\mathbb{R}^n)$ to $L^1(\mathbb{R}^n)$. 
    \item{\textnormal{(b)}}  Let $\beta=n\varepsilon$ where  $\varepsilon_0\leq \varepsilon<\frac{Q_0}{2n}.$
    Then, $A$ admits a bounded extension from $L^p(\mathbb{R}^n)$ to itself for all $1<p<\infty.$
   \item{\textnormal{(c)}}  For $0\leq \beta<n\varepsilon_0,$ $A$ admits a bounded extension from $L^p(\mathbb{R}^n)$ to itself provided that $$\beta\geq 2n\varepsilon_0\left|\frac{1}{2}-\frac{1}{p}\right|.$$ 
\end{itemize}
\end{theorem} 
\subsubsection{Open questions}
Nevertheless, Theorem \ref{Th:Delgado:2005} leaves open the question about whether the weak (1,1) inequality holds for the classes $S(m^{-\tau},g) $ and for which sharp order $\tau$.  Moreover, it is also an open problem to determine if additional properties are required for the weight  $m=(a+\langle\xi\rangle)^\frac{1}{2}.$  Since $a$ has Kohn-Nirenberg order $m=2,$ one indeed has the estimate $m\lesssim \langle \xi\rangle.$ However, in order to have better control of Littlewood-Paley components, we will consider the following assumption, namely, that the following {\it hypoellipticity } condition must hold
\begin{equation}\label{kappa:1:2}
    \exists \kappa\in [1,2], \quad \langle \xi\rangle^{\frac{\kappa}{2}}\lesssim m.
\end{equation}Note that this particularly condition holds when $a\geq 0,$ is such that 
\begin{equation}\label{hypoellipticity}
    \exists \kappa'\in [0,2], \quad \langle \xi\rangle^{\kappa'}\lesssim a.
\end{equation}The case $\kappa'=2$ in the previous assumption \eqref{hypoellipticity} is exactly the ellipticity condition on $a:=a(x,\xi),$ and the non-superfluous case corresponds to the range $\kappa'\in [1,2),$ since for $\kappa'\in [0,1),$ one would have $\langle\xi\rangle\lesssim m,$ that is, \eqref{kappa:1:2} holds for $\kappa=1.$ 
In this paper with the goal of having a better understanding of these questions, we have proved the following theorem.
\subsubsection{The weak (1,1) type of operators}
Here we recall that a collection of vector fields satisfies the H\"ormander Condition at step $r\in \mathbb{N}$ if these vectors together with their commutators up to length $r$ span $\mathbb{R}^n$ at each point.
\begin{theorem}\label{theo:weak-(1,1):intro}
  Let $\sigma \in S(m^{-\tau},g)$ and  assume that for some $\kappa\in [1,2],$ the weight $m$ satisfies the condition \eqref{kappa:1:2}. Let us consider that the system $\mathbb{X}=\{X_j:1\leq j\leq k\}$ in (\ref{X:UHC}) satisfies the {\it{Uniform H\"ormander Condition (UHC)}} at step $r$. Assume that the following order inequality
  \begin{equation*}\label{eq:kernel-estimate-weak-R<1:intro}     
      \tau \ge \max\left\{ \frac{n(4-\kappa^2)-r_0(4-2\kappa)}{2\kappa^2}, \frac{n}{\kappa}\left(1-\frac{\kappa}{2r}\right)\right\}.
  \end{equation*}
  holds. Then, the pseudo-differential operator $\sigma(X, D)$ is bounded from $L^1(\mathbb{R}^n)$ to $L^{1,\infty}(\mathbb{R}^n).$
\end{theorem}
\subsubsection{Boundedness on Hardy spaces}
As for the boundedness properties from the Hardy space $H^p_\varrho,$ adapted to a suitable  we prove the following regularity result. Below $Q$ corresponds to the homogeneous dimension associated to the system of vector fields $\mathbb{X}$, see Remark \ref{Q:definition}.
\begin{theorem}\label{Hardy:Hp:Lp:intro}  Let us consider that the system $\mathbb{X}=\{X_j:1\leq j\leq k\}$ in (\ref{X:UHC}) satisfies the {\it{Uniform H\"ormander Condition (UHC)}} at step $r$. 
    Let $\sigma \in S(m^{-\tau},g)$  and let 
    \begin{equation*}
        \frac{n}{2}\left(1-\frac{n\kappa}{2Q}\right) \le \beta. 
    \end{equation*}
    Assume that 
    \begin{equation*}
        \tau \ge \max \left\{\frac{n(4-\kappa^2)-r_0(4-2\kappa)}{2\kappa^2},  \frac{2}{\kappa}\beta\right\}.
    \end{equation*}
     Then, $\sigma(x, D)$ is bounded from $H^p_\varrho$ to $L^p$ for $p_0\le p\le1,$ always that
     \begin{equation*}
         \frac{1}{p_0}=\frac{1}{2} + \frac{\beta Q(1/n+\kappa/4) }{Q+\beta  Q\kappa /2 - n\kappa/2  }, \quad \textnormal{if } \beta<\frac{n}{2};
     \end{equation*}
     \begin{equation*}
         \frac{1}{p_0}=\frac{1}{2} + \frac{ Q(1/2+n\kappa/8) }{Q+n  Q\kappa /4 - n\kappa/2  }, \quad \textnormal{if }  \beta\ge\frac{n}{2}.
     \end{equation*}
     If $\kappa=2$ and $Q=n$, then $\sigma(x, D)$ is bounded from the standard Hardy space $H^p_\varrho=H^p$ to $L^p$ for all $p$ in the range $p_0< p\le1,$ and 
     \begin{equation*}
         p_0=\frac{n}{n+1}.
     \end{equation*}
\end{theorem}
\begin{figure}[htpb]
    \centering
    \begin{tikzpicture}
        \begin{axis}[
            width=10cm,
            height=7cm,
            xlabel={$\kappa$},
            ylabel={$p_0$},
            xmin=1, xmax=2,
            ymin=0.5, ymax=1.0,
            xtick={1,2},
            ytick={1/2, 3/4, 1},
            yticklabels={$1/2$, $3/4$, 1},
            samples=50,
            grid=both,
            grid style={line width=.1pt, draw=gray!20},
            major grid style={line width=.2pt,draw=gray!50}
        ]
        
        \addplot[blue, thick, domain=1.001:1.999] {1/(0.5 + 3/2*(6-1)*(1/3 + x/4)/(6-1 + 3/2*(6-1)*x/2-3*x/2)
        )
        };
        \addplot[blue, thick, domain=1.001:1.999] {1/(0.5 + 3/2*(6-2)*(1/3 + x/4)/(6-2 + 3/2*(6-2)*x/2-3*x/2)
        )
        };
        \addplot[blue, thick, domain=1.001:1.999] {1/(0.5 + 3/2*(6-3)*(1/3 + x/4)/(6-3 + 3/2*(6-3)*x/2-3*x/2)
        )
        };
        \end{axis}
    \end{tikzpicture}
    \caption{Critical exponents $p_0$ as a function of $\kappa$ for $r_0=1, 2, 3,$ with $r=2,$ and $n=3$. Notice each curve approaches the corresponding value of $n/(n+1)$ as $\kappa\to2,$ when $r_0=n$ }
    \label{fig:critical-p}
\end{figure}
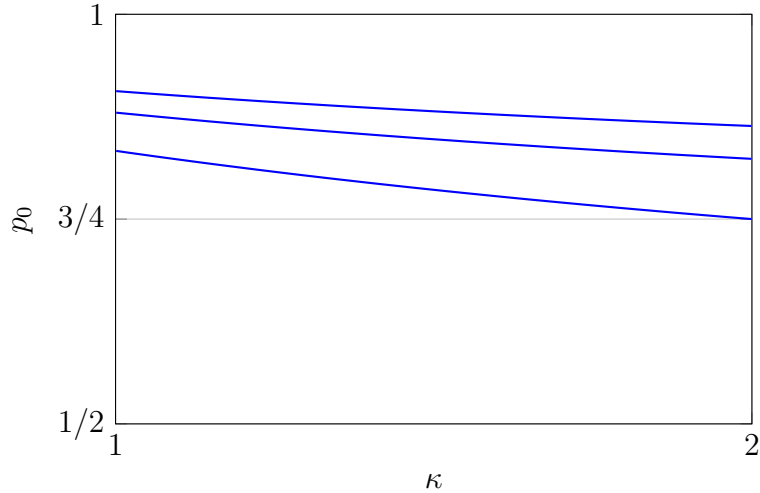
\subsubsection{Boundedness for Sobolev and Besov spaces}
Moreover, we include the following boundedness results for Sobolev and Besov spaces. 

\begin{theorem}\label{Lorentz:Sobolev:Besov:intro}
Let $\sigma\in S(m^{-\tau},g)$. Let $s,s'\in\mathbb{R},$ and let $1<p<\infty.$ Assume that 
      \begin{equation*}
            \tau \ge s'-s+ (n-r_0)\left|\frac{1}{2}-\frac{1}{p}\right|.
        \end{equation*}
        Then, we have that $\sigma(x, D)$ extends to a bounded operator from $H^s_p(m, g)$ to $H^{s'}_p(m, g).$ Moreover, for $1<p<\infty,$ we have that $\sigma(x, D)$ extends to a bounded operator from $B^s_{p,q}(m,g)$ to $B^{s'}_{p,q}(m,g).$
\end{theorem}
For regularity properties and the theory of pseudo-differential operators on Lie 
groups, particularly in the non-commutative case, we refer the reader to
\cite{Cardona:2026,cardona2026weak,CardonaDelgadoRuzhansky:Lp:Pseudo,
Cardona:Martinez:2025a,Coriasco:Toft2016,cardona-ruzhansky-subelliptic,
delgado-ruzhansky,ruzhansky-turunen}. For the aspects related to the Weyl-H\"ormander calculus we will follow  H\"ormander \cite{Hormander1985III}.

\subsection{Structure of the manuscript} This paper is organized as follows:

$\bullet$ In Section \ref{section:prelims} we provide the preliminaries about sub-Riemannian geometry used in this paper, the Weyl quantization and the Weyl-H\"ormander calculus. 

We also discuss the Uniform H\"ormander Condition $(UHC)$ identified here as a crucial hypothesis in the analysis of Littlewood-Paley components. We end with the definition of Hardy spaces adapted to a subelliptic distance.

$\bullet$ The proofs of our main theorems are presented in Section \ref{section:main}. Indeed, we start with Subsection \ref{Geometric:properties} by presenting the geometric properties induced to the volume between level surfaces $m(z,\xi)\asymp t,$ defined by the weigh function $m:=m(\cdot,\cdot).$ 

We also investigate the effect of the $UHC$-property into the subelliptic distance that was chosen in the definition of Hardy spaces. Then the regularity results for the $S(m,g)$-classes are investigated, starting with the boundedness on Sobolev and Besov spaces, on Lorentz spaces, including the weak (1,1) type of the corresponding operators.

\section{Preliminaries}
\label{section:prelims}
In this section we provide the notions related to the Weyl-H\"ormander classes $S(m, g),$ for this aspect we follow H\"ormander \cite{Hormander1985III}. We also present other notions required in our analysis. We start in the next section introducing the Uniform H\"ormander Condition, used as a hypothesis of our boundedness results. When there is $C>0$ such that $A\le CB,$ then we write $A\lesssim B.$ If, in addition $B\lesssim A,$ we write $A\asymp B.$ Moreover, we denote by $\lfloor x\rfloor$ the least integer $N,$ such that $N\ge x.$
\subsection{The Uniform H\"ormander Condition}\label{UHC:Subsec}

Let $A(x):=(a_{ij}(x))$ be a semi-positive definite matrix of rank $r(x)\ge r_0$ so that its entries and their derivatives are smooth and uniformly bounded, and let 
\begin{equation*}
    a(x, \xi) := \sum_{ij}a_{ij}(x)\xi_i\xi_j.
\end{equation*}
Let us define
\begin{equation}\label{metricgxx}
    g_X(\diff x, \diff \xi) := m^{-2}(X)\left(\langle\xi\rangle^2 \diff x^2 +\diff \xi^2\right), \quad m(X) := \sqrt{a(x, \xi)+\langle\xi\rangle}.
\end{equation}
Since $A(x)$ is semi-definite positive, it can be factorized as 
\begin{equation*}
    A(x) = B^*(x)B(x),
\end{equation*}
or equivalently, it can be written as 
\begin{equation*}
    A(x) = \sum_{j=1}^k X_j^*(x)X_j(x),
\end{equation*}
where each $X_j(x)$ is a smooth bounded vector field. We say that this collection of vector fields satisfy the \textit{H\"ormander Condition} at step $r\in \mathbb{N},$ if these vector fields together with their commutators up to length $r,$ span $\mathbb{R}^n$ at each point $x.$ Let us define 
\begin{align*}
    X^{(1)}&:= \{ X_1, \ldots, X_k\}, \\
    X^{(2)}&:= \{ [X_i, X_j] : 1\le i<j\le k\}, ...,
\end{align*}
so that the elements of $X^{(s)}$ are the commutators of length $s$. If $Y_j\in X^{(s)}$, then we say that it has degree $d_j=s.$ We say that these vectors satisfy the {\it Uniform H\"ormander Condition (UHC)}, if for any spanning subset of $n$ vector fields we have that 
\begin{equation*}
    |\det (Y_{j_1}, \ldots , Y_{j_n})| \ge c_0>0,
\end{equation*}
for every $x\in\mathbb{R}^n.$ Now, we define two notions of distance that take into account the geometric particularities of this problem. 
\begin{definition}\label{def:varrho}
    We say $\varphi\in C(\delta)$ if it is an absolutely continuous mapping $\varphi:[0, 1]\to \mathbb{R}^n$ such that 
    \begin{equation*}
        \varphi'(t) = \sum_j a_j(t)Y_j(\varphi(t)), \quad |a_j(t)|<\delta^{d_j}.
    \end{equation*}
    Moreover, we define
    \begin{equation*}
        \varrho(x, y):= \inf\{\delta>0:\exists \varphi\in C(\delta) , \, \varphi(0)=x,\, \varphi(1)=y\}.
    \end{equation*}
\end{definition}
\begin{definition}
    We say that $y\in B(x, R)$ if there exists a Lipschitzian path $\gamma(t)$, $0\le t\le R,$ such that $\gamma(0)=x,$ and $\gamma(R)=y$. Moreover, we require that it is subunitary with respect to $A(x)$, namely, that the tangent vectors
    \begin{equation*}
        \gamma'(t) =: \lambda(t),
    \end{equation*}
    satisfy for almost every $t\in [0,R] ,$ the inequality
    \begin{equation*}
        \sum_{ij}\lambda_i\lambda_j\xi_i\xi_j \le \sum_{ij} a_{ij}(\gamma(t)) \xi_i\xi_j.
    \end{equation*}
    We define the distance $d(\cdot,\cdot)$ below as 
    \begin{equation*}
        d(x, y) := \inf \{R>0: y\in B(x, R) \}.
    \end{equation*}
\end{definition}
These two distances are equivalent, see Nagel, Stein and Wainger \cite[Theorem~4]{nagel:stein:wainger}.
\begin{theorem}
    If the UHC is satisfied, then the distances $\varrho(\cdot, \cdot)$ and $d(\cdot, \cdot)$ are equivalent.
\end{theorem}
\begin{remark}\label{Q:definition}
    For the remainder of this paper, we are going to denote the sub-unitary balls as $B(z, R)$ and the Euclidean balls as $B_E(z, R).$ Moreover, we have that 
    \begin{equation*}
        B(z, R) \asymp R^Q, \quad Q:= d_1+\cdots+d_n.
    \end{equation*}
    Notice that when $r=2,$ then we have that 
    \begin{equation*}
        Q=r_0+2(n-r_0) = Q_0.
    \end{equation*}
    Moreover, it was proved in \cite{nagel:stein:wainger} that this disance satisfies the doubling condition
\begin{equation*}
    |B(z, 2R)| \le C|B(z, R)|.
\end{equation*}
\end{remark}
\subsection{The Weyl quantization}
In what follows we introduce the notion of pseudo-differential operator. We record the quantization process due to Kohn and Nirenberg \cite{KN:1965}. Also, we present the Weyl quantization. A motivation for this is that real-valued symbols are associated to symmetric operators.
\begin{definition}
    Let $\sigma(x, \xi) \in \mathcal{S}'(\mathbb{R}^{2n}).$ We define its \emph{Kohn-Nirenberg quantization} as the operator $\sigma(x, D):\mathcal{S}(\mathbb{R}^{n})\to\mathcal{S}(\mathbb{R}^{n})$ defined by
    \begin{equation*}
        \sigma(x, D)f(x) := \int_{\mathbb{R}^n}\int_{\mathbb{R}^n}e^{i(x-y)\cdot\xi}\sigma(x, \xi)f(y)\diff \xi \diff y.
    \end{equation*}
    Moreover its \emph{Weyl quantization} is defined as
    \begin{equation*}
        \sigma^wf(x) := \int_{\mathbb{R}^n}\int_{\mathbb{R}^n}e^{i(x-y)\cdot\xi}\sigma\left(\tfrac{x+y}{2}, \xi\right)f(y)\diff \xi \diff y.
    \end{equation*}
\end{definition}
According to the Weyl-H\"ormander classes, derivatives of a symbol are controlled by a (H\"ormander) metric $g$ and the weight $m$ provides some notion of order for the symbol. The family of symbols $S(m,g)$ is introduced below. The notion of H\"ormander metric can be introduced as follows.
\begin{definition}
    For $(x, \xi)\in \mathbb{R}^{2n} $, let $g_{(x, \xi)}(\cdot)$ be a positive definite quadratic form on $\mathbb{R}^{2n}$. We say $g$ is a H\"ormander metric if the following conditions are satisfied
    \begin{enumerate}
        \item \emph{Continuity.} There exists constants $C, c, c'>0$ such that $g_{(x,\xi)}(y, \eta) \le C $ implies that 
        \begin{equation*}
            c'g_{(x+y, \xi+\eta)}(t, \tau) \le g_{(x,\xi)}(t,\tau) \le c g_{(x+y, \xi+\eta)}(t, \tau)
        \end{equation*}
        for every $(t, \tau)\in \mathbb{R}^{2n}.$
        \item \emph{Uncertainty principle.} Let us define 
        \begin{equation*}
            g^\sigma_{(x, \xi)}(t, \tau) := \sup_{(z, \zeta)\ne0} \frac{(z\cdot \tau -t\cdot\zeta)^2}{g_{(x, \xi)}(z, \zeta)},
        \end{equation*}
        where $z\cdot\tau-t\cdot\zeta$ denotes the sympletic form in $\mathbb{R}^n.$
        We say $g$ satisfies the uncertainty principle if 
        \begin{equation*}
            \lambda_g(x, \xi) := \inf_{(t,\tau)\ne0} \left(\frac{g^\sigma_{(x,\xi)}(t, \tau)}{g_{(x,\xi)}(t, \tau)}
            \right)^{1/2} \ge 1,
        \end{equation*}
        for all $(x, \xi)\in \mathbb{R}^{2n} $
        \item \emph{Temperancy.} We say that $g$ is temperate if the exists $\overline{C}>0$ and $J\in \mathbb{N}$ such that 
        \begin{equation*}
            \left(\frac{g_{(x, \xi)}(\cdot)}{g_{(y, \eta)}(\cdot)}
            \right)^{\pm1} \le \overline{C} \left(1+g^\sigma_{(y,\eta)}(x-y, \xi-\eta)\right)^J.
        \end{equation*}
    \end{enumerate}
\end{definition}
Weights $M$ for the $S(M, g)$ classes are typically considered admissible with respect to the metric $g,$ as we record in the following definition.
\begin{definition}
    We say that a strictly positive function $M$ is a $g$-admissible weight if it satisfies the following
    \begin{enumerate}
        \item \emph{Continuity.} There exists $\tilde{C}>0$ such that $g_{(x, \xi)}(y,\eta)\le 1/\tilde{C}$ implies that
        \begin{equation*}
            \left(\frac{M(x+y, \xi+\eta)}{M(x, \xi)}\right)^{\pm1} \le \tilde{C}.
        \end{equation*}
        \item \emph{Temperancy.} There exists $\tilde{C}>0$ and $N\in\mathbb{N}$ such that 
        \begin{equation*}
            \left(\frac{M(y, \eta)}{M(x, \xi)}
            \right)^{\pm1} \le \tilde{C}\left(1+g^\sigma_{(x,\xi)}(x-y, \xi-\eta)\right)^N.
        \end{equation*}
    \end{enumerate}
\end{definition}
\subsection{The Weyl-H\"ormander calculus}
The calculus of pseudo-differential operator for the $S(M, g)$ classes extend several pseudo-differential calculus, including the classical $(\rho,\delta)$-H\"ormander classes, see \cite{Hormander1967}. For more details on this calculus, we refer the reader to \cite{Hormander1985III}.
\begin{definition}
    For a H\"ormander's metric $g$, and a $g$-admissible weight $M,$ we denote by $S(M, g)$ the set of all smooth functions $\sigma$ on $\mathbb{R}^{2n}$ such that for any integer $k,$ there exists $C_k>0$ so that for any $(x, \xi), (t_1, \tau_1),\ldots,(t_k,\tau_k),$ we have that the directional derivatives satisfy
    \begin{equation*}
        \left|\sigma^{(k)}(x, \xi)((t_1, \tau_1),\ldots,(t_k,\tau_k))
        \right| \le C_k M(x, \xi)\prod_{j=1}^k g_{(x, \xi)}^{1/2}(t_j, \tau_j).
    \end{equation*}
    We denote by $\|\sigma\|_{k;S(M,g)}$ the minimum $C_k$ that satisfies these estimates. When $M,g$ can be understood from the context, we will use $\|\sigma\|_{(k)}$.
\end{definition}
\begin{definition}
    For a H\"ormander's metric $g$, and a $g$-admissible weight $M,$ we define the Sobolev space relative to $M$, denoted by $H(M, g)$, as the set of tempered distributions $u$ on $\mathbb{R}^{n}$ such that 
    \begin{equation*}
        \sigma^wu \in L^2
    \end{equation*}
    for all $\sigma\in S(M, g).$
\end{definition}
\begin{remark}
    For the properties of the Sobolev spaces $H( M, g)$ we refer to Bony and Chemin \cite{BC:1994}. We will consider a general class of Sobolev and Besov spaces in subsection \ref{ss:Sobolev-Besov}.
\end{remark}
\begin{remark} (1) For $0\le\delta\le\rho\le 1$, with $\delta<1$, we have that the metric given by 
    \begin{equation*}
        g^{\rho,\delta}_{(x,\xi)} := \langle\xi\rangle^{2\delta}\diff x^2 + \langle\xi\rangle^{-2\rho}\diff \xi^2,
    \end{equation*}
    is a H\"ormander metric. In particular, the function given by
    \begin{equation*}
        \Lambda^m(x, \xi):= \langle\xi\rangle^m,
    \end{equation*}
    is a $g^{\rho,\delta}$-admissible weight for every $m\in\mathbb{R}$. Moreover, we have that
    \begin{equation*}
        S^m_{\rho,\delta} = S(\Lambda^m, g^{\rho,\delta}).
    \end{equation*}
    Here, we say that $\sigma \in S^m_{\rho,\delta}$ if the following inequalities
    \begin{equation}\label{eq:rho-delta-symbol}
        \left|\partial^\mu_\xi\partial^\nu_x \sigma(x, \xi)
        \right| \le C_{\mu\nu} \langle\xi\rangle^{m-\rho|\mu|+\delta|\nu|},
    \end{equation}
    hold, for every $\mu,\nu\in\mathbb{N}^n_0.$\\%

    (2) The metric defined in \eqref{metricgxx} is a H\"ormander metric and $m$ is a $g$-admissible weight (cf. \cite{c-c-x:se}). 
\end{remark}
\begin{remark} 
    Let $(M, g)$ be a Riemannian manifold of bounded geometry and let $\nabla$ be a linear connection with Christoffel symbols given by $\Gamma_{kj}^j(x).$ For $0\le\delta<\rho\le1,$ let us define the metric 
    \begin{equation*}
        {g}_{x, \xi}^{\rho,\delta,\nabla} := \langle\xi\rangle_x^{2\delta} |\diff x|^2 + \langle\xi\rangle^{-2\rho}\sum_j\left| \diff \xi_j -\sum_{i,k}\Gamma_{kj}^i(x)\xi_i\diff x^k
        \right|^2,
    \end{equation*}
    and the weight 
    \begin{equation*}
        \Lambda^m_g(x, \xi) := \langle\xi\rangle_x^m,
    \end{equation*}
    where 
    \begin{equation*}
        \langle\xi\rangle_x:= \sqrt{1+g^{ab}(x)\xi_a\xi_b}.
    \end{equation*}
    Then, the Safarov symbol class may be identified as
    \begin{equation*}
        S^m_{\rho,\delta}(\nabla) = S(\Lambda^m_g, g^{\rho,\delta,\nabla}).
    \end{equation*}
    We direct the reader to \cite{GomezCobosRuzhansky2025, CardonaGomezCobosRuzhanskyMartinezRuzhanskyYeghoyan} for more details.
\end{remark}
\begin{definition}
    A $g$-admissible weight $M$ is regular if $M\in S(M, g). $
\end{definition}
\begin{remark} For the H\"ormander metric defined in \eqref{metricgxx}, the weight $m$ is regular as it was shown in \cite{c-c-x:se}.
 
\end{remark}
 
The following Fefferman-type boundedness theorem can be deduced from a general version stated by the second author in \cite{Delgado2016}, see also \cite{Delgado2006}.
\begin{theorem}\label{Delgado:Theorem:2016}
    Let $A(x)=a_{ij}(x)$ be a semi-definite positive matrix with smooth and uniformly bounded coefficients such that $r(x):=\mathrm{rank}(A(x)) \ge r_0\ge1$, and it satisfies the UHC at step $r$. Let 
    \begin{equation*}
        m(x, \xi) = \sqrt{  A(x)\xi\cdot\xi + \langle\xi\rangle }, \quad g_{(x,\xi)}= m^{-2}(x,\xi)\left( \langle\xi\rangle^2\diff x^2+\diff \xi^2
        \right).
    \end{equation*}
    Let $\sigma\in S(m^{-\tau},g)$ and $1<p<\infty$. Then, $\sigma(X, D)$ is bounded from $L^p(\mathbb{R}^n)$ to itself if 
    \begin{equation*}
        \tau \ge (n-r_0)\left| \frac{1}{2}-\frac{1}{p} \right|.
    \end{equation*}
\end{theorem}
\subsection{Function spaces} 
In this subsection, we will include the definitions of the function spaces that will be studied in this paper.
\begin{definition}
    Let $0<p<\infty.$ We say that $u$ belongs in the Lorentz space $L^{p,q}:=L^{p,q}(\mathbb{R}^n),$ if for $0<q<\infty$, we have that
    \begin{equation*}
    \|u\|_{L^{p,q}} :=  \left( p\int_0^\infty \lambda^q \left|\left\{ x\in \mathbb{R}^n: |u(x)|>\lambda 
    \right\}\right|^{q/p} \frac{\diff t}{t}
    \right)^{1/q} < \infty,
\end{equation*}
    and for $q=\infty$, we have that 
    \begin{equation*}
    \|u\|_{L^{p,\infty}} :=  \sup_{\lambda>0}\lambda^q \left|\left\{ x\in \mathbb{R}^n: |u(x)|>\lambda 
    \right\}\right|  < \infty.
\end{equation*}
    It is a known fact that $L^{p,p}$ coincides with the Lebesgue space $L^p:=L^p(\mathbb{R}^n).$
\end{definition}
Triebel \cite[p.~134]{triebel} proved a real interpolation result for these spaces. We record it in the following statement.
\begin{proposition}\label{pro:lorentz-real-interpolation}
    Let $0<\theta<1$, let $1< p_0,p_1<\infty$ so that $p_0\ne p_1$. Let $1\le q_0,q_1,q\le\infty$ and let 
    \begin{equation*}
        \frac{1}{p}=\frac{1-\theta}{p_0} + \frac{\theta}{p_1}.
    \end{equation*}
    Then 
    \begin{equation*}
        (L^{p_0,q_0}, L^{p_1,q_1})_{\theta,q} = L^{p,q}
    \end{equation*}
    where $(\cdot,\cdot)_{\theta,q}$ is the real interpolation functor.
\end{proposition}
On the other hand, let $\varrho(\cdot,\cdot)$ be the sub-Riemannian distance as in Definition \ref{def:varrho}. Since the distance $\varrho(\cdot,\cdot)$ satisfies the doubling condition,  we may use the atomic construction of Hardy spaces as in Coifman and Weiss \cite[p.~591]{CoifmanWeiss1977}.
\begin{definition}
    For $0< p<q,$ and $p\le 1\le q\le \infty.$ we say a complex valued function $b$ is a $(p,q)$-\emph{atom} related to the sub-unitary ball $B(z, R)$ if it satisfies the following properties:
    \begin{enumerate}
        \item $\supp b \subset B(z, R) ,$
        \item $\|b\|_{L^q} \le  R^{Q(1/q-1/p)},$ 
        \item $\int b(x)\diff x =0.$
    \end{enumerate}
    We say that $f\in H^{p,q}_\varrho:=H^{p,q}_\varrho(\mathbb{R}^n), $ if  there exist a sequence of $(p,q)$-atoms $b_j$ and a sequence $\lambda_j\in\ell^p$ such that $u$ can be decomposed, namely
    \begin{equation}\label{eq:decomposition}
        u= \sum_j\lambda_jb_j.
    \end{equation}
    Moreover, we define its $H^{p,q}_\varrho$ norm by taking the infimum of 
    \begin{equation*}
        \left(\sum_j|\lambda_j|^p
        \right)^{1/p}
    \end{equation*}
    over all such decompositions as in \eqref{eq:decomposition}. In \cite[Theorem~A]{CoifmanWeiss1977}, it was proved that the spaces $H^{p,q_1}_\varrho$ and $H^{p,q_2}_\varrho$ are equivalent. Thus, $q$ may be chosen as convenient in each particular situation.
\end{definition}
\begin{remark}
    Hardy spaces may also be defined using maximal functions as in Mac\'ias and Segovia \cite{MaciasSegovia}. First, recall that 
    \begin{equation*}
        |B(x, R)| \asymp R^{Q}.
    \end{equation*}
    Thus, we have that 
    \begin{equation*}
        \left|\varrho(x, z)^{Q}-\varrho(y, z)^{Q}\right| \le C(R^{Q})^{1-1/Q}(\varrho(x, y)^{Q})^{1/Q},
    \end{equation*}
    whenever $x,y\in B(z, R^{1/Q}).$ Namely, we have that our space is a normal space of degree $1/Q,$ see \cite{MaciasSegovia}. Now, we define the Lipschitz space $\mathrm{Lip}(\beta)$ as the set of functions $\phi$ such that 
    \begin{equation*}
        |\phi(x)-\phi(y)|\le C\varrho(x, y)^{\beta Q}.
    \end{equation*}
    We denote the smallest constant $C$ that satisfies this condition as $\|\phi\|_\beta.$ Moreover, we say that $\phi \in T_\gamma(x)$ for $0<\gamma<1/Q,$ if $\phi\in\mathrm{Lip}(\beta)$ for all $\beta\le 1/Q,$ if there exists $R>0$ such that $\supp \phi \subset B(x, R),$ and 
    \begin{equation*}
        R^{Q}\|\phi\|_{L^\infty}\le1, \quad R^{(1+\gamma)Q}\|\phi\|_\gamma \le 1.
    \end{equation*}
    We define the $\gamma$-maximal function of a suitable distribution as (see \cite{MaciasSegovia})
    \begin{equation*}
        f^*_\gamma(x) := \sup \{|\langle f,\phi\rangle|:\phi\in T_\gamma(x)\}.
    \end{equation*}
    In \cite[Theorem~5.9]{MaciasSegovia}, it was proved that a function $f\in H^p_\varrho,$ if and only if, $f^*_\gamma \in L^p,$ whenever $1/(1+\gamma)< p\le 1$. Moreover, we have that 
    \begin{equation*}
        \|f\|_{H^p} \asymp \|f^*_\gamma\|_{L^p}.
    \end{equation*}
    Thus, we require that $p>Q/(Q+1).$
\end{remark}

\subsection{Auxiliary results}
Now, we state the $L^2$-boundedness Theorem by Hounie \cite{hounie}. This result is interesting as it provides an extension of the classical Calderón-Vaillancourt Theorem, see \cite{CV}.
\begin{theorem}\label{theo:hounie-L2}
    Let $a:=a(x,\zeta)$ be a function in $\mathbb{R}^n\times\mathbb{R}^n$ so that for $|\alpha|,|\beta|\leq \lfloor n/2+1\rfloor$ the derivatives $\partial^\alpha_\zeta \partial^\beta_x a$ are continuous and satisfy 
    \begin{equation*}
        \left|\partial^\alpha_\zeta \partial^\beta_x a(x,\zeta)\right| \leq C\langle\zeta\rangle^{m-\rho|\alpha|+\delta|\beta|},
    \end{equation*}
    where $0\leq\delta<1$, $0<\rho\leq1$, so that 
    \begin{equation*}
        m \leq -n\max \{ 0, (\delta-\rho)/2 \}.
    \end{equation*}
    Then $a(X, D)$ is bounded in $L^2(\mathbb{R}^n)$ with norm proportional to the best bound $C$.
\end{theorem}
Next we record the result from \'Alvarez and Milman \cite{alvarez-milman}, which we will use to prove weak (1,1) boundedness for a pertinent class $S(m^{-\tau},g)$ after obtaining the required kernel estimates. 
\begin{theorem}\label{theo:alvarez-milman}
    Let $T:C^\infty(\mathbb{R}^n)\to \mathcal{D}'(\mathbb{R}^n)$ be an integral kernel operator that extends to a bounded operator from $L^2(\mathbb{R}^n)$ to itself and from $L^q(\mathbb{R}^n)$ to $L^2(\mathbb{R}^n)$ such that for some $\alpha$ and $\beta$, we have that 
    \begin{equation*}
        \frac{1}{q}=\frac{1}{2} + \frac{\beta}{n}, \quad \frac{n}{2}(1-\alpha)\le \beta < \frac{n}{2}.
    \end{equation*}
    Moreover, assume that its kernel $k(x, y)$ satisfies the following estimates
    \begin{equation}\label{eq:alvarez-milman-R<1}
        \forall R<1, \sup_{|x-y|<R} \int_{|x-z|>cR^\alpha} |k(x, y)-k(x,z)|\diff x \le C.
    \end{equation}
    \begin{equation}\label{eq:alvarez-milman-R>1}
        \forall R\ge1,\sup_{|x-y|<R} \int_{|x-z|>cR} |k(x, y)-k(x,z)|\diff x \le C.
    \end{equation}
    Then, the operator $T$ is bounded from $L^1(\mathbb{R}^n)$ to $L^{1,\infty}(\mathbb{R}^n).$
\end{theorem}
Finally, we include the boundedness properties of Bessel potentials.
\begin{theorem}
    Let us define the Bessel potential of order $s\in\mathbb{R}$, denoted by $\mathfrak{B}^s,$ as the pseudo-differential operator with symbol given by $\langle\xi\rangle^s.$ Let $1<p<q<\infty,$ such that 
    \begin{equation*}
        s = -n\left(\frac{1}{q}-\frac{1}{p}\right).
    \end{equation*}
    Then, we have that $\mathfrak{B}^s$ is bounded from $L^p$ to $L^q.$
\end{theorem}
\section{Regularity properties for $S(m, g)$-classes revisited}
\label{section:main}
In this section we investigate the mapping properties for pseudo-differential operators for the classes $S(m,g).$ The next subsection will be dedicated to the analysis of some geometric estimates.
\subsection{Geometric properties}\label{Geometric:properties}
First, we are interested in the size of the phase space. We follow the strategy of \cite[Appendix~1]{Delgado2016}
\begin{lemma}\label{lem:phase-space-size}
    For $t\ge 1$ and $z\in\mathbb{R}^n$, assume that 
  \begin{equation*}
      |a(x, \xi)+\langle\xi\rangle| \ge C\langle\xi\rangle^{\kappa},
  \end{equation*}
  for $\kappa \in [1,2].$ Then, we have that
    \begin{equation*}
        \int_{m(z, \xi)\asymp t} \diff \xi \leq Ct^{Q_\kappa},
    \end{equation*}
    with $C>0$ independent of $z$ and with $Q_\kappa:=r_0 + \frac{2}{\kappa}(n-r_0)$.
\end{lemma}
\begin{proof}
    Recall that $a_{ij}(z)$ are the entries of a positive semi-definite matrix of rank $r(z)\ge r_0$. Hence, there exists a non-singular matrix $\theta(z)$ such that 
    \begin{equation*}
        \sum_{i,j}a_{ij}(z)\xi_i\xi_j = \sum_{j=1}^{r(x)} \lambda_j(z)(\theta(x)\xi)_j^2,
    \end{equation*}
    where $\lambda_j$ are the eigenvalues of the matrix $a_{ij}$. Since this matrix is of rank at least $r_0$, then there exists $c>0$ such that 
    \begin{equation*}
        \lambda := \inf_{z\in\Omega}\sup_{J\in\mathcal{J}(z)}\prod_{j\in J} \lambda_j\ge c,
    \end{equation*}
    where 
    \begin{equation*}
        \mathcal{J}(z):= \{ J=(j_1,\ldots,j_{r_0}): 1\le j_1 \le \cdots \le j_{r_0} \le r(z)\}.
    \end{equation*}
    Then, for a fixed $z_0$, we may fix the same positive eigenvalues $\lambda_j(z)$ for a neighborhood of $z_0$. Moreover, we define 
    \begin{align*}
        \overline{\xi} &:= \theta(z)\xi ,\\
        \|\overline{\xi}\| &:= \left(\sum_{j=1}^{r_0}\overline{\xi}^{4/\kappa}_j + \sum_{j=r_0+1}^n \overline{\xi}_j^2\right)^{\kappa/4},\\
        \delta_t\overline{\xi}&:= (t\overline{\xi}_1,\ldots,t\overline{\xi}_{r_0},t^{2/\kappa}\overline{\xi}_{r_0+1}, \ldots, t^{2/\kappa}\overline{\xi}_{n}).
    \end{align*}
    From the properties of $\theta$ and $\lambda$, we have that
    \begin{equation*}
        c(1+\|\overline{\xi}\|) \le m(z, \theta^{-1}(z)\overline{\xi})\le Ct.
    \end{equation*}
    By setting, $\overline{\xi}:=\delta_t\xi'$, we have that 
    \begin{equation*}
        \|\xi'\|\le C.
    \end{equation*}
    Now, let us define $\xi'=h(\zeta)$, where
    \begin{equation*}
        h(\zeta):= \left(\zeta_1, \ldots, \zeta_{r_0},\zeta_{r_0+1}|\zeta_{r_0+1}|^{2/\kappa-1},\ldots, \zeta_n|\zeta_n|^{2/\kappa-1}\right)
    \end{equation*}
    to obtain that
    \begin{equation*}
        |\zeta|\le C_1.
    \end{equation*}
    Notice that the Jacobian of the transformations to obtain $\zeta$ from $\xi$ is 
    \begin{equation*}
        |\theta^{-1}(z)|t^{Q_\kappa}\left(\frac{2}{\kappa}\right)^{n-r_0}|\zeta_{r_0+1}\cdots\zeta_n|^{2/\kappa-1}.
    \end{equation*}
    On the other hand, we have that 
    \begin{align*}
        m(z, \theta^{-1}(z)\overline{\xi}) &\le C \left( \sum_{j=1}^{r_0}\overline{\xi}_j^2+\sum_{j=r_0+1}^{r(z)}\lambda_j(z)\overline{\xi}_j^2 + \sum_{j=r(z)}^n |\overline{\xi}_j|^\kappa \right)^{1/2} ,
    \end{align*}
    and for $\overline{\xi}:=\delta_{t}\xi'$, we have 
    \begin{align*}
        C & \le  \left( \sum_{j=1}^{r_0}(\xi'_j)^2+\sum_{j=r_0+1}^{r(z)}t^{4/\kappa-2}\lambda_j(z)(\xi'_j)^2 + |\xi'|^\kappa \right)^{1/2}.
    \end{align*}
    Now, let us set
    \begin{equation*}
        \xi''_j := \begin{cases}
            \xi'_j, & j\le r_0\text{ or } t^{4/\kappa-2}\lambda_j \le 1; \\
            t^{2/\kappa-1}\lambda_j(x)^{1/2}\xi'_j, & \text{otherwise.}
        \end{cases}
    \end{equation*}
    Thus, we have that
    \begin{equation*}
        C  \le  \left( \sum_{j=1}^{r_0}(\xi''_j)^2+\sum_{j=r_0+1}^{r(z)}(\xi''_j)^2 + |\xi''|^\kappa \right)^{1/2} = C \left( \sum_{j=1}^{r(z)}(\xi''_j)^2 + |\xi''|^\kappa \right)^{1/2}.
    \end{equation*}
    Finally, setting $\xi''=h(\zeta)$, 
    we obtain that
    \begin{equation*}
        C_{r(z_0)} \le |\zeta|.
    \end{equation*}
    Notice the Jacobian of this transformation is comparable to
    \begin{align*}
        &|\theta^{-1}(z)|t^{Q_\kappa}\left[\prod_{t^{4/\kappa-2}\lambda_j\ge 1} t^{1-2/\kappa}\lambda_j(z)^{-1/2}\right] \left(\frac{2}{\kappa}\right)^{n-r_0}|\zeta_{r_0+1}\cdots\zeta_n|^{2/\kappa-1}\\
        & \le |\theta^{-1}(z)|t^{Q_\kappa} \left(\frac{2}{\kappa}\right)^{n-r_0}|\zeta_{r_0+1}\cdots\zeta_n|^{2/\kappa-1}.
    \end{align*}
    Since $r(z_0)$ takes a finite number of values, we have that $|\zeta| \asymp 1$ independent of $z$. Therefore, we have that
    \begin{align*}
        \int_{m(z,\xi)\asymp t} \diff \xi & \le C \int_{|\zeta|\asymp 1}|\theta^{-1}(z)|t^{Q_\kappa} \left(\frac{2}{\kappa}\right)^{n-r_0}|\zeta_{r_0+1}\cdots\zeta_n|^{2/\kappa-1} \diff \zeta \\
        & \le Ct^{Q_\kappa}\int_{|\zeta|\asymp 1} |\zeta_{r_0+1}\cdots\zeta_n|^{2/\kappa-1} \diff \zeta \\
        &\le Ct^{Q_\kappa},
    \end{align*}
    where we used the fact that $|\zeta_j|\le |\zeta|\le C$ and the integral converges since the exponent satisfies $2/\kappa-1 \ge 0$.
\end{proof}
In the previous proof, we diagonalised the operator at the point $z\in\mathbb{R}^n.$ That means that we may rewrite $\xi=\xi'+\xi''$, where $\xi''$ is in the kernel of $A(z).$ We may use the dual projection in the space variable. Thus, we may rewrite, $y-z=\Theta(y, z)'+\Theta(y,z)'',$ such that 
\begin{equation*}
    (y-z)\cdot\xi = \Theta(y,z)'\cdot \xi' + \Theta(y, z)''\cdot\xi''.
\end{equation*}
We will be using this notation when convenient.
\begin{lemma}
    Let $R>0,$ if $\varrho(y, z)= R,$ then we have that 
    \begin{equation*}
        |\Theta(y,z)''| \le CR^2.
    \end{equation*}
\end{lemma}
\begin{proof}
    Let $\gamma(s)$ be a sub-unitary curve from $z$ to $y$, and let us define $\lambda(s):= \dot{\gamma}(s)$. Recall that we require that
    \begin{equation*}
        ( \lambda(s)\cdot \xi)^2 \le A(\gamma(s))\xi \cdot \xi ,
    \end{equation*}
    for almost every $s.$ If we set $\lambda=\xi$, we obtain that
    \begin{equation*}
        |\lambda|^4 \le A(\gamma(s))\lambda \cdot \lambda \le \|A(\gamma(s))\||\lambda|^2 \le C_1 |\lambda|^2,
    \end{equation*}
    since $a_{ij}(y)$ are bounded functions. This implies that 
    \begin{equation}\label{eq:distance-bound-linear}
        |\gamma(s)-z| \le \int_0^s |\lambda(s')|\diff s' \le C_1s.
    \end{equation}
    On the other hand, let $\xi''$ be in the kernel of $A(z)$ so that $|\xi''|= 1$. Moreover, recall that $f_{\xi''}(z)=A(z)\xi''\cdot\xi''=0$, and it must be a minimum, meaning that $\nabla_z f_{\xi''}(z) = 0$. Then, by Taylor's theorem we have that 
    \begin{align*}
        f_{\xi''}(w) &= f_{\xi''}(z) + \nabla_zf_{\xi''}(z)(w-z) + (w-z)\cdot H(A(z)\xi''\cdot\xi'')(w')(w-z) \\
        & = (w-z)\cdot H_{f_{\xi''}}(w')(w-z),
    \end{align*}
    for some $w'$ in the line between $w$ and $z$. Notice that we have the inequality 
    \begin{equation*}
        \sup_{w\in \mathbb{R}^n, |\xi''|=1} \left\|H_{f_{\xi''}}(w) \right\| \le C_2,
    \end{equation*}
    since the entries of $A(w)$ are smooth and bounded. Combining these estimates, we deduce that 
    \begin{equation*}
        (\lambda(s)\cdot\xi'')^2 \le A(\gamma(s))\xi''\cdot\xi'' \le C_2|\gamma(s)-z|^2 \le C_1^2C_2s^2,
    \end{equation*}
    holds uniformly in $\xi''.$ Finally, the following estimates
    \begin{equation*}
        |(y-z)\cdot \xi''| \le \int_0^R|\lambda(s)\cdot\xi''|\diff s \le C^1C_2^{1/2}\int_0^R s \diff s \le \tfrac{1}{2 }C^1C_2^{1/2}R^2.
    \end{equation*}
    hold. Notice that the bounding constant above only depends on the boundedness of the functions $a_{ij}(w).$ Thus, we finally have that
    \begin{equation*}
        |\Theta''| \le \sup_{|\xi''|=1} |(y-z)\cdot\xi''| \le CR^2,
    \end{equation*}
    completing the proof.
\end{proof}
\begin{remark}
    Notice that (\ref{eq:distance-bound-linear})  shows that we have the inequality
    \begin{equation*}
        |y-z|\le C\varrho(y,z),
    \end{equation*}
    uniformly in $y,z.$
\end{remark}
Now, we adapt a proof from \cite{nagel:stein:wainger} to obtain the converse result.
\begin{lemma}\label{lem:distance-comparison}
    If $a_{ij}(x)$ satisfy the Uniform H\"ormander Condition (UHC) at step $r$, then the Euclidean distance can be compared to $\varrho$ via 
    \begin{equation}\label{eq:distance-comp}
        C_1|x-y|\le \varrho(x, y) \le C_2 \max\{|x-y|^{1/r}, |x-y|\}, 
    \end{equation}
    uniformly.
\end{lemma}
\begin{proof}
    The left hand side of \eqref{eq:distance-comp} was already proved. For the right hand side, let $Y_1(x), \ldots, Y_m(x)$ be a collection of vectors that span $\mathbb{R}^n$ for every $x$, each one of degree $d_j$, and such that 
    \begin{equation*}
        \det (Y_1, \ldots,Y_m) \ge c_0 >0.
    \end{equation*}
    Then, let us define 
    \begin{equation*}
        \gamma(t):= ty + (1-t)x.
    \end{equation*}
    Notice that $|\dot{\gamma}(t)|\le |x-y|$. Now, let us write 
    \begin{equation*}
        \dot{\gamma}(t)=\sum b_j(t)Y_j(\gamma(t)).
    \end{equation*}
    To obtain these coefficients $b_j(t)$, we must solve the linear system 
    \begin{equation*}
        \dot{\gamma}(t) = (Y_1, \ldots, Y_m)(\gamma(t))b(t),
    \end{equation*}
    thus, 
    \begin{equation*}
        |b_j(t)|\le \|(Y_1, \ldots, Y_m)(\gamma(t))^{-1}\||\dot{\gamma}(t)| \le C(c_0)|x-y| \le C(|x-y|^{1/d_j})^{d_j}.
    \end{equation*}
    The result follows from the fact that $d_j \le r$
\end{proof}
\begin{lemma}\label{lem:kernel-decay}
    If $\varrho(x, y)>1$, then
    \begin{equation*}
        |k(x, y)|\le C\|\sigma\|\varrho(x, y)^{-N}.
    \end{equation*}
    If $\varrho(x, y)<1$, then
    \begin{equation*}
        |k(x, y)| \le C\|\sigma\| \varrho(x, y)^{-N/r}.
    \end{equation*}
\end{lemma}
\begin{proof}
    We have that 
    \begin{align*}
        (x-y)^\alpha k(x, y) = \int_{\mathbb{R}^n}e^{i(x-y)\cdot\xi}D^\alpha_\xi \sigma(x, \xi) \diff \xi.
    \end{align*}
    In particular, for $|\alpha|=N$, we have that 
    \begin{align*}
        |x-y|^N|k(x, y)| & \le\int_{\mathbb{R}^n}\|\sigma\|m^{-\tau - N}(x, \xi) \diff \xi \\
        & \le \int_{\mathbb{R}^n} \|\sigma\|\langle\xi\rangle^{-(\tau + N)/2} \diff \xi \\
        & \le C\|\sigma\|,
    \end{align*}
    when $(\tau+N)/2 > n.$ The result follows from the distance comparison in Lemma \ref{lem:distance-comparison}.
\end{proof}
\subsection{Boundedness on Sobolev spaces and Besov spaces}\label{ss:Sobolev-Besov} In this section we provide boundedness theorems for pseudo-differential operators with symbols in $S(m,g)$-classes on a scale of Sobolev and of Besov spaces adapted to the Weyl-H\"ormander classes for which we follow the approach due to J. M. Bony and Y. Chemin, see \cite{BC:1994}.
\subsubsection{ A class of Sobolev spaces and Besov spaces}
As above, let $A(x)=a_{ij}(x)$ be a semi-definite positive matrix with smooth and uniformly bounded coefficients such that $r(x):=\mathrm{rank}(A(x)) \ge r_0\ge1$, and it satisfies the UHC at step $r$. Let 
    \begin{equation*}
        m(x, \xi) = \sqrt{  A(x)\xi\cdot\xi + \langle\xi\rangle }, \quad g_{(x,\xi)}= m^{-2}(x,\xi)\left( \langle\xi\rangle^2\diff x^2+\diff \xi^2
        \right).
    \end{equation*} Since $m^2\in S^2_{1,0},$ in view of Fefferman-Phong inequality, see \cite{FP:1978}, we have that for some $C>0,$ the following lower bound property
    \begin{align*}
        Re(m^2(x,D)u,u)\geq -C\Vert u\Vert_{L^2},
    \end{align*} holds.  Using the Weyl-quantisation, the previous inequality is also equivalent to say that there exists a constant $C>0,$ such that 
    \begin{align*}
        Op^w(m^2)+CI\geq 0.
    \end{align*} The previous fact also implies that for some constant $c>0,$ the operator $\mathcal{R}:= Op^w(m^2)$ is such that
    \begin{equation}
        \mathcal{R}+cI= Op^w(m^2)+cI\geq 0,
    \end{equation} is an unbounded and densely defined positive operator on $L^2.$ The constant $c>0$ can be taken in such a way that $\lambda=0$ is an isolated point of the spectrum of $\mathcal{R}+cI.$ The operator $ \mathcal{R}+cI$ is unbounded and self-adjoint since 
    $$ (\mathcal{R}+cI)^*= Op^w(m^2)^*+cI=Op^w(m^2)+cI.$$
    So, let us consider the spectral measure $\{dE_\lambda\}$ of the operator $\mathcal{R}+cI.$ Let us consider the powers $(\mathcal{R}+cI)^\frac{s}{2}$ defined by the functional calculus of self-adjoint operators, namely, defined in terms of the operator valued integral
    \begin{equation}
        (\mathcal{R}+cI)^\frac{s}{2}:=\int_b^\infty\sqrt{\lambda}^s dE_\lambda,
    \end{equation} with the property that, since $\lambda=0$ can be assumed outside of the support of $\{dE_\lambda\},$ one can consider $b>0.$
We consider the following scale of Sobolev spaces:
\begin{definition} For any $s\in \mathbb{R},$ and $1\leq p\leq \infty,$ consider the Sobolev spaces $H^s_p(m,g)$ defined by the completion of $C^\infty_0(\mathbb{R}^n)$ with respect to the norm
\begin{equation}
    \Vert u\Vert_{H^s_p(m,g)}= \Vert (\mathcal{R}+cI)^\frac{s}{2} u\Vert_{L^p}.
\end{equation}Let us consider also the family of Besov spaces $ B^s_{p,q}(m,g)$ defined by the real-interpolation as
\begin{equation}
    B^s_{p,q}(m,g)=(H^{s_0}_p(m,g),H^{s_1}_p(m,g))_{\theta,q},
\end{equation} where $\theta\in [0,1]$ and $s_0$ and $s_1$ are a couple of real numbers such that $s=(1-\theta)s_0+\theta s_1,$ with the corresponding norm $\Vert \cdot\Vert_{ B^s_{p,q}(m,g)}$ induced by the interpolation functor.

\end{definition}
\begin{remark}
    When $p=2,$ we will write $H^s(m,g)\equiv H^s_2(m,g).$ These spaces were defined, e.g. in Bony and Chemin \cite{BC:1994} and with their nomenclature, the space  $H^s(m,g)$ has been denoted as  $H(m^s,g)$. Our little change in the notation will be justified by the fact that in the spaces $H^s_p(m,g),$ the parameters $s$ and $p,$ of regularity order and of integrability degree, respectively, play an essential role in interpolation techniques used for this paper. Note that in the case where $r_0=n,$ the Sobolev space $H^s_p(m,g),$ agrees with the Sobolev space $H^s_p(\mathbb{R}^n)$ on $\mathbb{R}^n,$ and the Besov spaces $B^s_{p,q}(m,g)$ recover the definition of the standard Sobolev spaces  $B^s_p(\mathbb{R}^n)$ on the Euclidean space.
\end{remark}
\begin{remark} Let us describe some properties of the operator $   \mathcal{R}+cI$ and of its inverse as pseudo-differential operators and as operators defined by the functional calculus. First of all, one has that
   \begin{equation}
        \mathcal{R}+cI= Op^w(m^2)+cI: H^2(m,g)\rightarrow L^2,
    \end{equation}is injective and bounded from below. In consequence it is surjective. Moreover it is an invertible pseudo-differential operator with pseudo-differential inverse, see Bony and Chemin \cite[Page 116]{BC:1994}. This means that there is $\mathfrak{m}^{-II}\in S(m^{-2},g) $ such that $$\mathfrak{m}^{-II}\#(m^2+c)=1,$$ and 
    \begin{equation}
        (\mathcal{R}+cI)^{-1}:L^2\rightarrow H^2(m,g)
    \end{equation} is a pseudo-differential operator with $$ (\mathcal{R}+cI)^{-1}:=Op^w(\mathfrak{m}^{-II}).$$
    Moreover, in the domain of $\mathcal{R}+cI,$ one also has that
\begin{equation}
        (\mathcal{R}+cI)^{-1}:=\int_b^\infty{\lambda}^{-1} dE_\lambda,
\end{equation}admits a bounded extension to $L^2$  and is self-adjoint.
\end{remark}
\begin{remark}\label{Membership:s:m:g:powers}
    In general, for any $s\in \mathbb{R},$ the operator $  (\mathcal{R}+cI)^\frac{s}{2}$ has a symbol (in the Weyl and also for the Kohn-Nirenberg quantisation) in the class $S(m^{s},g),$ see \cite[Page 16]{Delgado2016}.
\end{remark}
\begin{lemma}
    Let $s\in \mathbb{R},$ and let us consider $1<p<\infty.$ We have the following embedding properties
    \begin{equation}
      H^{-s}_p(m,g)  \hookrightarrow   L^p \hookrightarrow  H^s_p(m,g)
    \end{equation} provided that 
    $$s\leq -(n-r_0)\left|\frac{1}{2}-\frac{1}{p}\right|.$$
\end{lemma}
\begin{proof}
    In view of Theorem \ref{Delgado:Theorem:2016} and of Remark \ref{Membership:s:m:g:powers}, the {\it a priori} estimate
    \begin{equation}
        \Vert (\mathcal{R}+cI)^\frac{s}{2} u\Vert_{L^p}\leq C\Vert u\Vert_{L^p},\,\,u\in C^\infty_0,
    \end{equation} holds provided that $s\leq -(n-r_0)\left|\frac{1}{2}-\frac{1}{p}\right|.$ This proves the continuous inclusion $ L^p \hookrightarrow  H^s_p(m,g).$ By the duality argument one also has that  $H^{-s}_q(m,g)  \hookrightarrow   L^q$ where $q=\frac{p}{p-1}$ is the conjugate exponent of $p.$ Note also the order condition
    $$s\leq -(n-r_0)\left|\frac{1}{2}-\frac{1}{p}\right|=(n-r_0)\left|\frac{1}{2}-\left(1-\frac{1}{q}\right)\right|= -(n-r_0)\left|\frac{1}{2}-\frac{1}{q}\right|.$$
    Since $p$ is an arbitrary Lebesgue exponent in the range $(1,\infty),$ one also has that $q$ lies in the same range. The proof is complete.
\end{proof}

Then, we extend this result to the general case between two Sobolev spaces. This result is presented in Theorem \ref{Lorentz:Sobolev:Besov:intro} (b).
\begin{theorem}
    Let $s,s'\in\mathbb{R},$ and let $1<p<\infty.$ Let $\sigma\in S(m^{-\tau},g)$ and assume that 
    \begin{equation*}
        \tau \ge s'-s+ (n-r_0)\left|\frac{1}{2}-\frac{1}{p}\right|
    \end{equation*}
    Then, we have that $\sigma(x, D)$ extends to a bounded operator from $H^s_p(m, g)$ to $H^{s'}_p(m, g).$
\end{theorem}
\begin{proof}
    Notice that the operator $(\mathcal{R}+cI)^{s'/2}\sigma(x, D)(\mathcal{R}+cI)^{-s/2}$ belongs to the class $S(m^{s'-\tau-s}, g).$ Thus, in view of Theorem \ref{Delgado:Theorem:2016}, it is bounded in $L^p,$ and we have that
    \begin{align*}
        \|\sigma(x, D)u\|_{H^{s'}_p} & = \|(\mathcal{R}+cI)^{s'/2}\sigma(x, D)u\|_{L^p} \\
        & = \|(\mathcal{R}+cI)^{s'/2}\sigma(x, D)(\mathcal{R}+cI)^{-s/2}(\mathcal{R}+cI)^{s/2}u\|_{L^p} \\
        & \le C\|\sigma\|_{(l)}\|(\mathcal{R}+cI)^{s/2}u\|_{L^p} \\
        & =  C\|\sigma\|_{(l)}\|u\|_{H^s_p}.
    \end{align*}
    This completes the proof.
\end{proof}
Now, we use a real interpolation argument to obtain the following result for Besov spaces.
\begin{corollary}
    Let $s,s'\in\mathbb{R},$ and let $1<p<\infty.$ Let $\sigma\in S(m^{-\tau},g)$ and assume that 
    \begin{equation*}
        \tau \ge s'-s+ (n-r_0)\left|\frac{1}{2}-\frac{1}{p}\right|
    \end{equation*}
    Then, we have that $\sigma(x, D)$ extends to a bounded operator from $B^s_{p,q}(m, g)$ to $B^{s'}_{p,q}(m, g).$
\end{corollary}
\begin{proof}
    Note that $\sigma(x, D)$ is bounded from $H^{s_0}_p(m, g)$ to $H^{s_0'}_p(m, g)$ and from $H^{s_1}_p(m, g)$ to $H^{s_1'}_p(m, g)$ given that 
    \begin{equation*}
        \tau\ge s_j' - s_j + (n-r_0)\left|\frac{1}{2}-\frac{1}{p}\right|.
    \end{equation*}
    Thus, we use a real interpolation argument to obtain that $\sigma(x D)$ is bounded from $H^{s}_p(m, g)$ to $H^{s'}_p(m, g)$ for $s=(1-\theta)s_0 + \theta s_1,$  $s'=(1-\theta)s_0' + \theta s_1',$ and $0<\theta<1$
\end{proof}
\subsection{Weak (1,1) type and boundedness in Lorentz spaces}
Now, we prove the kernel estimates required for the proof of the weak (1,1) boundedness extending to the $S(m, g)$-setting the approach due to \'Alvarez and Hounie \cite{alvarez-hounie}.
\begin{theorem}\label{theo:kernel-estimate-first}
Let $A(x)=a_{ij}(x)$ be a semi-definite positive matrix with smooth and uniformly bounded coefficients such that $r(x):=\mathrm{rank}(A(x)) \ge r_0\ge1$, and assume that it satisfies the UHC at step $r$. Let 
    \begin{equation*}
        m(x, \xi) = \sqrt{  A(x)\xi\cdot\xi + \langle\xi\rangle }, \textnormal{ and } g_{(x,\xi)}= m^{-2}(x,\xi)\left( \langle\xi\rangle^2\diff x^2+\diff \xi^2
        \right).
    \end{equation*}
  Let $\sigma \in S(m^{-\tau},g)$. Assume that 
  \begin{equation*}
      |a(x, \xi)+\langle\xi\rangle| \ge C\langle\xi\rangle^{\kappa},
  \end{equation*}
  for some $\kappa \in [1, 2],$ and that the order condition
  \begin{equation*}
      \tau \ge   \frac{n(4-\kappa^2)-r_0(4-2\kappa)}{2\kappa^2},
  \end{equation*}
  holds. Then the Schwartz kernel $k:=k(x, y)$ of $\sigma(X, D)$ satisfies the following estimate:
  \begin{equation}\label{eq:kernel-estimate-weak-R<1}
    \sup_{\varrho(y,z)<R} \int_{|x-z|>2R^{\kappa/2}} |k(x,y)-k(x, z)|\diff x \leq C\|\sigma\|_{l;S},
  \end{equation}
  for any fixed $z\in \mathbb{R}^n$, whenever $0<R< 1$, and the inequalities
  \begin{equation}\label{eq:kernel-estimate-weak-R>1}
      \sup_{\varrho(y,z)<R} \int_{|x-z|>2R} |k(x,y)-k(x, z)|\diff x \leq C\|\sigma\|_{l;S},
  \end{equation}
  whenever $R\ge1.$
  The constant $C$ in \eqref{eq:kernel-estimate-weak-R>1} and in \eqref{eq:kernel-estimate-weak-R>1} does not depend on $z$. 
\end{theorem}
\begin{remark}
    When $\kappa=2$, we only require that $\tau\ge0$.
\end{remark}
\begin{proof}
    The proof of the case $R\ge1$ follows from the previous lemma. Let $\varphi \in C_0^\infty(\mathbb{R})$ be supported in $\left[\tfrac{1}{2}, 1\right]$ such that
  \begin{equation*}
    \int_0^\infty\varphi(1/t)\frac{\diff t}{t} =\int_1^2\varphi(1/t)\frac{\diff t}{t} = 1.
  \end{equation*}
  Define $\varphi_t(\xi):=\varphi(m(z,\xi)/t)$s o that in its support one has that $m(z, \xi)\asymp t$. Recall that 
  \begin{equation*}
      \langle\xi\rangle^{\kappa/2} \lesssim m(z, \xi) \lesssim \langle\xi\rangle,
  \end{equation*}
  so that the following inequalities hold
  \begin{equation*}
      \langle\xi\rangle^{\kappa/2} \lesssim t\lesssim \langle\xi\rangle,
  \end{equation*}
  or equivalently
  \begin{equation*}
      t\lesssim \langle\xi\rangle \lesssim t^{2/\kappa}.
  \end{equation*}
  Moreover, let us define $\xi=\xi'+\xi''$ where $\xi''$ is in the kernel of $A(z)$ and $\xi'$ is in its image space, so that 
  \begin{equation*}
       \langle\xi''\rangle^{\kappa/2}=m(z, \xi'') \asymp t, \quad \langle\xi'\rangle \asymp m(z, \xi') \asymp t.
  \end{equation*}
  Namely, the following estimates are satisfied
  \begin{equation*}
      \langle\xi'\rangle \asymp t, \quad \langle\xi''\rangle \asymp t^{2/\kappa}.
  \end{equation*}
  On the other hand, notice that $\partial^{\nu}_{\xi} \varphi_t(\xi)$ is equal to a linear combination of terms of the form
  \begin{equation*}
      t^{-k} \varphi^{(k)}(m(z, \xi)/t) \prod_{j=1}^k \partial^{\mu_j}_\xi m(z, \xi),
  \end{equation*}
  with $1\le k\le |\nu|$ and $\sum \mu_j = \nu.$ Moreover, each of these terms can be bounded by
  \begin{equation*}
      t^{-k}\prod_{j=1}^k\left|  \partial^{\mu_j}_\xi m(z, \xi)\right| \lesssim t^{-k}\prod_{j=1}^k m(z, \xi)^{1-|\mu_j|} \lesssim t^{-k} t^{k-\sum|\mu_j|} = t^{-|\nu|}.
  \end{equation*}
  Thus, we have that 
  \begin{equation*}
      \left|\partial^{\nu}_{\xi} \varphi_t(\xi)\right| \lesssim t^{-|\nu|}.
  \end{equation*}
   On the other hand, let us define the symbol $\sigma_t$ by
  \begin{equation*}
    \sigma_t(x, \xi) := \sigma(x, \xi)\varphi_t(\xi),
  \end{equation*}
  and let $k_t:=k_t(x, y)$ be its corresponding Schwartz kernel. Hence, we have that
  \begin{equation*}
    k(x, y) = \int_0^\infty k_t(x, y)\frac{\diff t}{t} = \int_1^\infty k_t(x, y)\frac{\diff t}{t},
  \end{equation*}
  since $m(z,\xi)\ge 1.$ Let $N = \lfloor n/2+1\rfloor $ be an integer, then we obtain the estimates
  \begin{align}
    &\int_{|x-z|>2R^{\kappa/2}} |k_t(x, y) - k_t(x, z)| \,\diff x\notag\\
    &\leq  \left[ \int_{|x-z|>2R^{\kappa/2}} (1+t^{2\rho}|x-z|^2)^{N} |k_t(x, y) - k_t(x, z)|^2 \,\diff x  \right]^{1/2}\notag  \\
    &\qquad \times\left[ \int_{\mathbb{R}^n} (1+t^{2\rho}|x-z|^2)^{-N} \,\diff x \right]^{1/2} \notag  \\
    &\leq C\left[ \int_{|x-z|>2R^{\kappa/2}} (1+t^{2\rho}|x-z|^2)^{N} |k_t(x, y) - k_t(x, z)|^2 \,\diff x  \right]^{1/2}t^{-\rho n/2}.
    \label{eq:firstbound-b}
  \end{align}
  Now, notice that we can rewrite
  \begin{align*}
      k_t(x, y) - k_t(x, z)
      & = \int_{\mathbb{R}^n}  e^{i(x-y)\cdot\xi} \sigma_t(x, \xi)\,\diff\xi - \int_{\mathbb{R}^n} e^{i(x-z)\cdot\xi}\sigma_t(x, \xi)\, \diff \xi \\ 
      & = \int_{\mathbb{R}^n}  e^{i(x-z)\cdot\xi} \left[ e^{i(z-y)\cdot\xi} -  1
      \right] \sigma_t(x, \xi) \diff \xi.
  \end{align*}
  For $|\alpha|\leq N$, Leibniz rule allows us to have the equality
  \begin{align*}
      &t^{\rho|\alpha|}\int_{\mathbb{R}^n} e^{i(x-z)\cdot\xi}D^{\alpha}_{\xi}\left[\left( e^{i(z-y)\cdot\xi} -  1
      \right) \sigma_t(x, \xi)\right] \diff \xi \\
      & = t^{\rho|\alpha|} \sum_{\beta+\gamma\le\alpha}C_{\beta\gamma}\int_{\mathbb{R}^n} e^{i(x-z)\cdot\xi}D^{\beta}_{\xi}\left[ e^{i(z-y)\cdot\xi} -  1
      \right]  D^{\gamma}_{\xi} \sigma(x, \xi) D^{\alpha-\beta-\gamma}_{\xi}\varphi_t(\xi) \diff \xi,
  \end{align*}
  which is a sum of pseudo-differential operators on $\mathbb{R}^n$ associated with the symbols 
  \begin{equation*}
      \sigma_t^{\beta\gamma}(x, \xi):= D^{\beta}_{\xi}\left[ e^{i(z-y)\cdot\xi} -  1
      \right] D^{\gamma}_{\xi} \sigma(x, \xi) \chi(\xi),
  \end{equation*}
  and where the property $\chi(\xi)\varphi_t(\xi) = \varphi_t(\xi)$ holds in view of the supports of $\chi$ and $\varphi_t$. Notice that we we have the estimates
  \begin{align*}
      \left| e^{i(z-y)\cdot\xi}-1 \right| \leq |(z-y)\cdot\xi|\le |\Theta'||\xi'| +|\Theta''||\xi''| \le  C(tR + t^2R^2) \le CtR,
  \end{align*}
  when $tR<1$, and that
  \begin{align*}
      \left|\partial^\nu_\xi D^{\beta}_{\xi} e^{i(z-y)\cdot\xi}\right| & \le |y-z|^{|\beta+\nu|} \\
      & \le C(tR) t^{-|\beta|}\langle\xi\rangle^{-|\nu|/2},
    \end{align*}
  when $tR<1$ since $\langle\xi\rangle\lesssim t^2 $. Moreover, we have that
  \begin{align*}
      \left|\partial^\nu_\xi\partial_x^\mu D^{\gamma}_{\xi} \sigma(x, \xi)\right| & \leq C\|\sigma\|_{l;S}m(X)^{-\tau}m(X)^{-|\gamma+\mu+\nu|}\langle \xi\rangle^{|\mu|} \\
      & \leq C\|\sigma\|_{l;S} m(X)^{-\tau-|\gamma|}\langle\xi\rangle^{-|\mu+\nu|/2}\langle\xi\rangle^{|\mu|} \\
      & \leq C\|\sigma\|_{l;S} t^{-\kappa\tau/2-\kappa|\gamma|/2} \langle\xi\rangle^{-|\nu|/2+|\mu|/2},
  \end{align*}
  since $ t^{\kappa/2}\lesssim m(X).$ Hence, we obtain that
    \begin{align*}
        \left|\partial_\xi^\nu\partial_x^\mu \sigma_t^{\beta\gamma}(x,\xi)\right| & \leq CtR\|\sigma\|_{l;S} t^{-\kappa\tau/2-\kappa|\gamma|/2}t^{-|\beta|}\langle\xi\rangle^{-|\nu|/2+|\mu|/2} \\ 
        & = CtR\|\sigma\|_{l;S}t^{-\kappa\tau/2-\kappa|\gamma|/2-|\beta|}\langle\xi\rangle^{-|\nu|/2+|\mu|/2}.
    \end{align*}
    Thus, one can deduce that $\sigma_t^{\beta\gamma}$ satisfies (\ref{eq:rho-delta-symbol}) for $\rho=\delta=1/2,$ and that $\sigma_t^{\beta\gamma} \in S\left(1, g^{1/2,1/2}\right),$ when $tR<1$. In particular, the symbols satisfy 
  \begin{equation*}
      \|\sigma^{\beta\gamma}_t\|_{l-|\gamma|;S} \leq Ctr\|\sigma\|_{l;S}t^{-\kappa\tau/2-\kappa|\gamma|/2-|\beta|}.
  \end{equation*}
  Now, by Theorem \ref{theo:hounie-L2} we have that the operators in the summation above induce bounded operators in $L^2(\mathbb{R}^n)$. By the Plancherel identity we may estimate (\ref{eq:firstbound-b}) when  $tr<1$ by
  \begin{align*}
      & Ct^{-\rho n/2}  \sum_{|\alpha|\leq N} t^{\rho |\alpha|} \sum_{\beta+\gamma\leq\alpha}(tR)\|\sigma\|_{l;S}t^{-\kappa\tau/2-\kappa|\gamma|/2-|\beta|} \left\| e^{-iz\cdot\xi} D^{\alpha-\beta-\gamma}_\xi\varphi_t(\xi) \right\|_{L^2_\xi}  \\ 
      & \leq Ct^{-\rho n/2}  \sum_{|\alpha|\leq N} t^{\rho |\alpha|} \sum_{\beta+\gamma\leq\alpha}(tR)\|\sigma\|_{l;S}t^{-\kappa\tau/2-\kappa|\gamma|/2-|\beta|} t^{-|\alpha-\beta-\gamma|}t^{Q_\kappa/2} \\
      & \leq   \sum_{|\alpha|\leq N} \sum_{\beta+\gamma\leq\alpha}(tR)\|\sigma\|_{l;S}t^{-\kappa\tau/2-\rho n/2+Q_\kappa/2} t^{\rho|\alpha|-\kappa|\gamma|/2-|\alpha-\gamma|}.
  \end{align*}
  Here we used the fact that the volume of the support of $\varphi_t$ is comparable to $t^{Q_\kappa}$ with $Q_\kappa=r_0 + \frac{2}{\kappa}(n-r_0),$ see Lemma \ref{lem:phase-space-size}. Now, we may make the choice of $\rho$ so that the second exponent is non-positive by
  \begin{align*}
    \rho|\alpha|-\kappa|\gamma|/2-|\alpha-\gamma|
    =(\rho-1)|\alpha| + (1-\kappa/2)|\gamma| 
     \le (\rho-\kappa/2)|\alpha|.
  \end{align*}
  Thus, we choose $\rho=\kappa/2$ and we must have that the following estimate holds
  \begin{equation*}
      -\kappa\tau/2-\kappa n/4+Q_\kappa/2 \le 0.
  \end{equation*}
  On the other hand, by the triangle inequality,  we have that $|x-y|>R^{\kappa/2}, $ and that
  \begin{align}
    &\int_{|x-y|>R^{\kappa/2}} |k_t(x, y)| \,\diff x\notag\\
    &\leq  \left[ \int_{|x-y|>R^{\kappa/2}} (t^{\kappa}|x-y|^2)^{N} |k_t(x, y)|^2 \,\diff x  \right]^{1/2}\notag  \\
    &\qquad \times\left[ \int_{|x-y|>R^{\kappa/2}} (t^\kappa|x-y|^2)^{-N} \,\diff x \right]^{1/2} \notag  \\
    &\leq C\left[ \int_{|x-y|>R^{\kappa/2}} (t^\kappa|x-y|^2)^{N} |k_t(x, y)|^2 \,\diff x  \right]^{1/2}t^{-\kappa N/2}R^{(n/2-N)\kappa/2}.
    \label{eq:secondbound-b}
  \end{align}
  As before, for $|\alpha|=N$, we have that 
  \begin{align*}
      &t^{\kappa|\alpha|/2} (x-y)^\alpha  \int_{\mathbb{R}^n} e^{i(x-y)\cdot\xi} \sigma(x, \xi)\varphi_t(\xi)\,\diff\xi \\
      & =  t^{\kappa|\alpha|/2} \sum_{\beta\leq\alpha}C_{\alpha\beta}  \int_{\mathbb{R}^n}e^{i(x-y)\cdot\xi} D^\beta_\xi \sigma(x, \xi) D^{\alpha-\beta}_\xi \varphi_t(\xi)\, \diff \xi.
  \end{align*}
   Notice this is a sum of pseudo-differential operators in $\mathbb{R}^n$ with symbols 
  \begin{equation*}
      \sigma^\beta_t(x,\xi):=D^\beta_\xi \sigma(x, \xi)\chi_t(\xi),
  \end{equation*}
  which satisfy (\ref{eq:rho-delta-symbol}) for $\rho=\delta=1/2.$ Namely,  $\sigma_t^{\beta} \in S\left(1, g^{1/2,1/2}\right),$  with norms
  \begin{equation*}
      \|\sigma^\beta_t\|_{l-|\beta|;S} \le t^{-\kappa\tau/2 - \kappa|\beta|/2}\|\sigma\|_{l;S}.
  \end{equation*}
  Moreover, by Theorem \ref{theo:hounie-L2}, it is bounded in $L^2(\mathbb{R}^n)$, and we may estimate (\ref{eq:secondbound-b}) using Plancherel by 
  \begin{align*}
      & Ct^{- \kappa N/2}R^{(n/2 - N)\kappa/2} \sum_{|\alpha|\leq N} t^{\kappa|\alpha|/2} \sum_{\beta\leq\alpha} t^{-\kappa\tau/2 - \kappa|\beta|/2}\|\sigma\|_{l;S} \left\| e^{-iy\cdot\xi}D^{\alpha-\beta}_{\xi} \varphi_t(\xi)
      \right\|_{L^2_\xi} \\ 
      &\leq Ct^{- \kappa N/2}R^{(n/2 - N)\kappa/2} \sum_{|\alpha|\leq N} t^{\kappa|\alpha|/2} \sum_{\beta\leq\alpha} t^{-\kappa\tau/2 - \kappa|\beta|/2}\|\sigma\|_{l;S} t^{-|\alpha-\beta|}t^{Q_0/2} \\
      &\leq C(tR)^{(n/2-N)\kappa/2}\|\sigma\|_{l;S} t^{-\kappa\tau/2-\kappa n/4+Q_0/2} \\
      &\leq C(tR)^{(n/2-N)\kappa/2}\|\sigma\|_{l;S},
  \end{align*}
  because of the choice of $\tau$ as in the previous estimate. Using the same argument, we can estimate
  \begin{equation*}
    \int_{|x-z|>2R^{\kappa/2}} |k_t(x, z)| \,\diff x \leq C\|\sigma\|_{l;S}(tR)^{(n/2 - N)\kappa/2},
  \end{equation*}
  by also writing it as a sum of pseudo-differntial operators in $\mathbb{R}^n$. Using these estimates, we get the result from the following expression:
  \begin{align*}
     \int_{\varrho(x, z)> 2R^{\kappa/2}}|k(x, y) - k(x, z)|\,\diff x 
    &\leq C\|\sigma\|_{l;S}\left[ \int_1^{1/R} tR\frac{\diff t}{t} + \int_{1/R}^\infty(tR)^{(n/2-N)\kappa/2}\frac{\diff t}{t}\right]\\
    &\leq C\|\sigma\|_{l;S}\left[R(1/R-1) + \frac{1}{(N-n/2)\kappa/2} \right] \\ 
    &\leq C\|\sigma\|_{l;S}.
  \end{align*}
  Thus, obtaining the result.
\end{proof}
Here, we use the properties of Bessel potentials to obtain a $L^q$-$L^2$-boundedness result.
\begin{theorem}\label{theo:Lq-L2}
    Let $\sigma\in S(m^{-\tau},g)$, and $1<q<2$. Then, $\sigma(X, D)$ is bounded from $L^q(\mathbb{R}^n)$ to $L^2(\mathbb{R}^n)$ if 
    \begin{equation*}
        \tau \ge \frac{2n}{\kappa}\left| \frac{1}{2}-\frac{1}{q} \right|.
    \end{equation*}
\end{theorem}
\begin{proof}
    Let us define
    \begin{equation*}
        s:= n\left(\frac{1}{q}-\frac{1}{2}\right),
    \end{equation*}
    so that $\mathfrak{B}^{-s}$ is bounded from $L^q$ to $L^2.$ Since $s\ge0,$ we have that 
    \begin{equation*}
        \langle\xi\rangle^s \lesssim m(x, \xi)^{2s/\kappa},
    \end{equation*}
    for every $x\in\mathbb{R}^n.$ Moreover, the following symbolic inequalities are satisfied
    \begin{equation*}
        \left|\partial^\nu_\xi\partial^\mu_x\langle\xi\rangle^s\right| \lesssim \langle\xi\rangle^{s-|\nu|} \lesssim m(x,\xi)^{2(s-|\nu|)/\kappa} \le m(x, \xi)^{2s/\kappa -|\nu|}.
    \end{equation*}
    Thus, we may conclude that $\mathfrak{B}^s\in S(m^{2s/\kappa},g).$  Then, we have that $\sigma(x, D)\mathfrak{B}^s$ is bounded in $L^2$ whenever $2s/\kappa-\tau\le 0.$ Thus, we obtain that
    \begin{align*}
        \|\sigma(x, D)\mathfrak{B}^s\mathfrak{B}^{-s}u\|_{L^2} & \le C\|\sigma\|_{(l)}\|\mathfrak{B}^{-s}u\|_{L^2} \\
        & \le C\|\sigma\|_{(l)}\|u\|_{L^q}.
    \end{align*}
     This completes the proof.
\end{proof}
Finally, we obtain the desired boundedness result.
\begin{theorem}\label{theo:weak-(1,1)}
  Let $\sigma \in S(m^{-\tau},g)$. Assume that 
  \begin{equation*}
      |a(x, \xi)+\langle\xi\rangle| \ge C\langle\xi\rangle^{\kappa},
  \end{equation*}
  for $\kappa \in [1, 2],$ and that
  \begin{equation*}\label{eq:kernel-estimate-weak-R<1}
      \tau \ge \max\left\{ \frac{n(4-\kappa^2)-r_0(4-2\kappa)}{2\kappa^2}, \frac{n}{\kappa}\left(1-\frac{\kappa}{2r}\right)\right\}.
  \end{equation*}
  Then, the operator $\sigma(X, D)$ is bounded from $L^1(\mathbb{R}^n)$ to $L^{1,\infty}(\mathbb{R}^n).$
\end{theorem}
\begin{proof}
    This is a direct application of Theorem \ref{theo:alvarez-milman}. Notice that when $R\ge1,$ the estimates (\ref{eq:alvarez-milman-R>1}) and (\ref{eq:kernel-estimate-weak-R>1}) are equivalent since $|y-z|\lesssim \varrho(x, y).$ For the case $R<1,$ notice that the requirement $|y-z|<R$ from (\ref{eq:alvarez-milman-R<1}) implies $\varrho(y, z)<CR^{1/r}$. Thus, we may choose 
    \begin{equation*}
        \alpha = \frac{\kappa}{2r}.
    \end{equation*}
    Indeed, evaluating in (\ref{eq:kernel-estimate-weak-R<1}) gives us 
    \begin{equation*}
        \sup_{\varrho(y,z)<R^{1/r}} \int_{\varrho(x, z)>2R^{\kappa/2r}} |k(x,y)-k(x, z)|\diff x \leq C\|\sigma\|_{(l)}.
    \end{equation*}
    Finally, we set 
    \begin{equation*}
        \beta=\tfrac{n}{2}(1-\kappa/2r),
    \end{equation*}
    and apply Theorem \ref{theo:Lq-L2} to obtain the $L^q$-$L^2$ boundedness.
\end{proof}
\subsection{$H^p_\varrho$-$L^p$-Boundedness}
Let us define 
\begin{equation*}
    A_j(z, R) := \{ x \in \mathbb{R}^n : 2^jR \le |x-z| < 2^{j+1}R\}. 
\end{equation*}
Now, we present some ``dyadic" kernel estimates as in \cite{alvarez-hounie}.
\begin{theorem}\label{theo:kernel-estimate-dyadic}
  Let $\sigma \in S(m^{-\tau},g)$. Assume that 
  \begin{equation*}
      |a(x, \xi)+\langle\xi\rangle| \ge C\langle\xi\rangle^{\kappa},
  \end{equation*}
  for $\kappa \in [1, 2],$ and that
  \begin{equation*}
      \tau \ge \frac{n(4-\kappa^2)-r_0(4-2\kappa)}{2\kappa^2}.
  \end{equation*}
  Then the Schwartz kernel $k:=k(x, y)$ of $\sigma(X, D)$ satisfies the following estimate for any fixed $z\in \mathbb{R}^n$:
  \begin{equation*}
    \sup_{\varrho(y,z)<R} \int_{A_j(z, R^\gamma)} |k(x,y)-k(x, z)|\diff x \leq C\|\sigma\|_{l;S}2^{-2j/\kappa}R^{1-2\gamma/\kappa},
  \end{equation*}
  whenever $0<R< 1$ and $\gamma\le\kappa/2$. Moreover, we have that 
  \begin{equation*}
      \sup_{\varrho(y,z)<R} \int_{A_j(z, R)} |k(x,y)-k(x, z)|\diff x \leq C\|\sigma\|_{l;S}2^{-j},
  \end{equation*}
  when $R\ge1.$ The constant $C$ does not depend on $z$. 
\end{theorem}
\begin{proof}
    The proof of the case $R\ge1$ comes from Lemma \ref{lem:kernel-decay}, so we may assume that $R<1$ for the rest of the proof. We define the symbol $\sigma_t$ as in the proof of Theorem \ref{theo:kernel-estimate-first}. Let $N>n/2$ be an integer to be precised later, then we obtain the estimates
  \begin{align*}
    &\int_{A_j(z, R^\gamma)} |k_t(x, y) - k_t(x, z)| \,\diff x\notag\\
    &\leq  \left[ \int_{\varrho(x, z)>2R^{\kappa/2}} (1+t^{\kappa}|x-z|^2)^{N} |k_t(x, y) - k_t(x, z)|^2 \,\diff x  \right]^{1/2}\notag  \\
    &\qquad \times\left[ \int_{A_j(z, R^\gamma)} (1+t^{\kappa}|x-z|^2)^{-N} \,\diff x \right]^{1/2} \notag  \\ 
    & \le CtR\|\sigma\|_{l;S}t^{\kappa n/4}\left[ \int_{A_j(z, R^\gamma)} (1+t^{\kappa}|x-z|^2)^{-N} \,\diff x \right]^{1/2},
  \end{align*}
  when $tR<1$, in view of the proof of Theorem \ref{theo:kernel-estimate-first} and the fact that $2R^\gamma\lesssim \varrho(x, z)$.
    To estimate the second factor, consider the function
    \begin{equation*}
        F(r)=\left(\int_{r}^{2r}(1+s^2)^{-N}s^{n-1}\diff{s}\right)^{1/2},
    \end{equation*}
    for $0<r<\infty.$ Note that $F(r)$ is a smooth function, that $F(r)\lesssim r^{n/2}$ as $r\to0$ and that $F(r)\lesssim r^{n/2-N}$ as $r\to\infty$. Hence, we have that 
    \begin{align*}
        &\left[ \int_{A_j(z, R^\gamma)} (1+t^{\kappa}|x-z|^2)^{-N} \,\diff x \right]^{1/2}\\
        &\le \left[ \int_{2^jR^\gamma\le |x-z|\le 2^{j+1}R^\gamma} (1+t^{\kappa}|x-z|^2)^{-N} \,\diff x \right]^{1/2}\\
        & \le C\left[ t^{-\kappa n/2}\int_{t^{\kappa/2}2^jR^\gamma}^{2(t^{\kappa/2}2^jR^\gamma)} (1+s^2)^{-N} s^{n-1} \,\diff s \right]^{1/2}\\
        & \le C t^{-\kappa n/4} F(t^{\kappa/2} 2^jR^\gamma) ,
    \end{align*}
    where we used the change of variables $s=t^{\kappa/2}|x-z|.$ On the other hand, in the proof of Theorem \ref{theo:kernel-estimate-first} it was proved that 
    \begin{align*}
        \sup_{\varrho(y, z)<R}\int_{|x-z|>2(2^{2j/\kappa} R^{2\gamma/\kappa} )^{\kappa/2}} |k_t(x, y)|+|k_t(x, z)|\diff x &\le C\|\sigma\| (t^{\kappa/2} 2^{j}R^{\gamma})^{n/2-N}.
    \end{align*}
    Combining these two estimates we obtain that
    \begin{align*}
        & \sup_{\varrho(y,z)<R} \int_{A_j(z, R^\gamma)} |k(x,y)-k(x, z)|\diff x \\
        & \le C\|\sigma\|_{(l)} \left[ \int_1^{1/R} tRF(t^{\kappa/2}2^jR^\gamma) \frac{\diff t}{t} + \int_{1/R}^\infty (t^{\kappa/2} 2^{j}R^{\gamma})^{n/2-N} \frac{\diff t}{t}
        \right] \\
        &  \le C\|\sigma\|_{(l)}\left[  R2^{-2j/\kappa}R^{-2\gamma/\kappa} + (2^{(N-n/2)\kappa/2 })^{-2j/\kappa} (R^{(N-n/2)\kappa/2})^{1-2\gamma/\kappa}
        \right] \\
        & \le C\|\sigma\|_{(l)}2^{-2j/\kappa}R^{1-2\gamma/\kappa}
    \end{align*}
    where we chose $N$ so that $(N-n/2)\kappa/2 >1.$
\end{proof}
\begin{theorem}
    Let $\sigma \in S(m^{-\tau},g)$  and let 
    \begin{equation*}
        \frac{n}{2}\left(1-\frac{n\kappa}{2Q}\right) \le \beta .
    \end{equation*}
    Assume that 
    \begin{equation*}
        \tau \ge \max \left\{\frac{n(4-\kappa^2)-r_0(4-2\kappa)}{2\kappa^2}, \frac{2}{\kappa}\beta \right\}.
    \end{equation*}
     Then, $\sigma(x, D)$ is bounded from $H^p_\varrho$ to $L^p$ for $p_0\le p\le1$ and 
     \begin{equation*}
         \frac{1}{p_0}=\frac{1}{2} + \frac{\beta Q(1/n+\kappa/4) }{Q+\beta  Q\kappa /2 - n\kappa/2  }, \quad \beta<\frac{n}{2};
     \end{equation*}
     \begin{equation*}
         \frac{1}{p_0}=\frac{1}{2} + \frac{ Q(1/2+n\kappa/8) }{Q+n  Q\kappa /4 - n\kappa/2  }, \quad \beta\ge\frac{n}{2}.
     \end{equation*}
     If $\kappa=2$ and $Q=n$, then $\sigma(x, D)$ is bounded from $H^p_\varrho$ to $L^p$ for $p_0< p\le1$ and 
     \begin{equation*}
         p_0=\frac{n}{n+1}.
     \end{equation*}
\end{theorem}
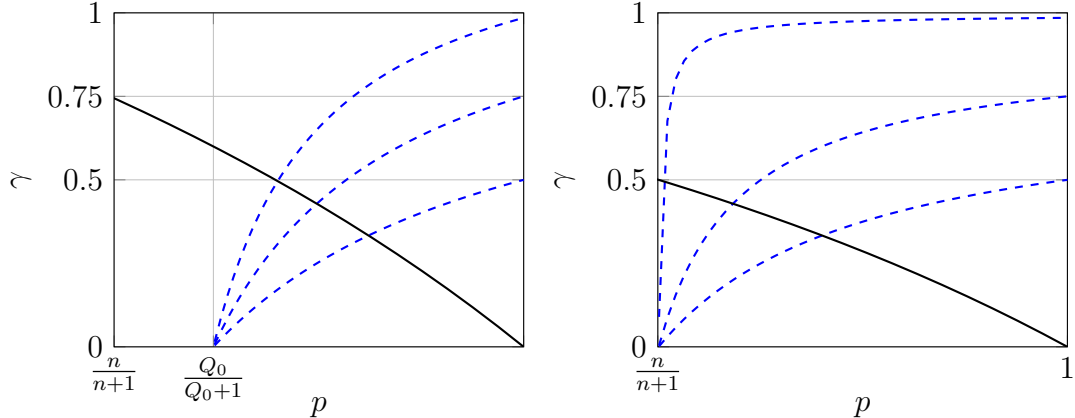
\begin{figure}[htpb]
    \centering
    \begin{tikzpicture}
        \begin{axis}[
            width=7cm,
            height=6cm,
            xlabel={$p$},
            ylabel={$\gamma$},
            xmin=0.67, xmax=1,
            ymin=0, ymax=1.0,
            xtick={0.67, 0.75, 1},
            xticklabels={$\frac{n}{n+1}$,$\frac{Q_0}{Q_0+1}$,,1},
            ytick={0, 0.5, 0.75, 1},
            samples=50,
            grid=both,
            grid style={line width=.1pt, draw=gray!20},
            major grid style={line width=.2pt,draw=gray!50}
        ]
        
        \addplot[blue, thick, dashed, domain=0.751:1] {(x*(5-1)-4+1)/(x*(2+2/1)-2)};

        \addplot[blue, thick, dashed, domain=0.751:1] {(x*(5-1)-4+1)/(x*(2+2/1.5)-2)};

        \addplot[blue, thick, dashed, domain=0.751:1] {(x*(5-1)-4+1)/(x*(2+2/1.97)-2)};

        \addplot[black, thick, domain=0:1] {(4-1)*(2-2*x)/(2*(2-x))};

        \end{axis}
    \end{tikzpicture}
    \begin{tikzpicture}
        \begin{axis}[
            width=7cm,
            height=6cm,
            xlabel={$p$},
            ylabel={$\gamma$},
            xmin=0.666, xmax=1,
            ymin=0, ymax=1.0,
            xtick={0.666,1},
            xticklabels={$\frac{n}{n+1}$,1},
            ytick={0, 0.5, 0.75, 1},
            samples=50,
            grid=both,
            grid style={line width=.1pt, draw=gray!20},
            major grid style={line width=.2pt,draw=gray!50}
        ]
        
        \addplot[blue, thick, dashed, domain=0.667:1] {(x*(5-2)-4+2)/(x*(2+2/1)-2)};

        \addplot[blue, thick, dashed, domain=0.667:1] {(x*(5-2)-4+2)/(x*(2+2/1.5)-2)};

        \addplot[blue, thick, dashed, domain=0.667:1] {(x*(5-2)-4+2)/(x*(2+2/1.97)-2)};

        \addplot[black, thick, domain=0:1] {(4-2)*(2-2*x)/(2*(2-x))};

        \end{axis}
    \end{tikzpicture}
    \caption{Upper and lower bounds for the parameter $\gamma,$ a requirement for convergence, as a function of $p \in \left(\frac{n}{n+1}, 1\right)$ with $n=2$. The lowest lower bound (when $\beta \to n/2$) is shown in solid black, while the upper bounds for $\kappa=1, 3/2$ and as $\kappa\to 2$ are shown in dashed blue. The first figure considers the case $r_0=1$ and the second one the case $r_0=2$. In all cases the upper bound goes to $\kappa/2$ as $p\to 1.$}
    \label{fig:gamma-bounds}
\end{figure}
\begin{proof}
    Let $b$ be a $(p,\infty)$-atom supported in the ball $B(z,R)$. First assume that $R<1,$ and $\beta<n/2,$ then we may rewrite
    \begin{align*}
        \int_{\mathbb{R}^n} |\sigma(x, D)b(x)|^p\diff x&\le \int_{B_E(z,2R^\gamma)} |\sigma(x, D)b(x)|^p\diff x + \sum_j \int_{A_j(z,R^\gamma)} |\sigma(x, D)b(x)|^p\diff x\\
        & = I_1 + I_2,
    \end{align*}
    with 
    \begin{equation*}
        I:= \int_{B_E(z,2R^\gamma)} |\sigma(x, D)b(x)|^p\diff x, \quad I_2:=\sum_j \int_{A_j(z,R^\gamma)} |\sigma(x, D)b(x)|^p\diff x.
    \end{equation*}
    Let us define
    \begin{equation*}
        \frac{1}{q}=\frac{1}{2}+\frac{\beta}{n},
    \end{equation*}
    so that $\sigma(x, D)$ is bounded from $L^q$ to $L^2$ in view of Theorem \ref{theo:Lq-L2}. Now, let us estimate the first term using H\"older's inequality with exponent $2/p$ and the $L^q$-$L^2$ boundedness to obtain that
    \begin{align*}
        I_1 & \le  \left(\int_{B_E(z,2R^\gamma)} |\sigma(x, D)b(x)|^2\diff x
        \right)^{p/2} \left(\int_{B_E(z,2R^\gamma)}\diff x
        \right)^{(2-p)/2} \\
        & = \|\sigma(x, D)b\|_{L^2}^p |B_E(z, R^\gamma)|^{(2-p)/2} \\
        &\le C \|\sigma\|_{(l)}\|b\|_{L^q}^pR^{n\gamma(2-p)/2}\\
        &\le C \|\sigma\|_{(l)}\left(\int_{B(z,R)}R^{-Qq/p}\diff x
        \right)^{p/q}R^{n\gamma(2-p)/2} \\
        & \le C \|\sigma\|_{(l)} R^{-Q}R^{Qp/q} R^{n\gamma(2-p)/2} \\
        & = C \|\sigma\|_{(l)}R^{n\gamma(1-p/2)+Q(p/q-1)}.
    \end{align*}
    Since, $1/q=1/2+\beta/n$, we have that $I_1$ is bounded whenever
    \begin{equation}\label{eq:gamma-lower}
        \gamma \ge Q\frac{2-p(1+2\beta/n)}{n(2-p)}.
    \end{equation}
    This is a decreasing function of $p$ that reaches
    \begin{equation*}
        \gamma \ge Q\frac{1-2\beta/n}{n}
    \end{equation*}
    when $p=1$, thus, the lower bound on $\beta.$ On the other hand, we may use the cancellation property of $b(x)$ to obtain that
    \begin{align*}
        I_2 & \le \sum_j \int_{A_j(z,R^\gamma)} \left( \int_{B(z,R)} |k(x, y) - k(x,z)|\|b\|_{L^\infty}\diff y
        \right)^p\diff x\\
        & \le \sum_j \int_{A_j(z,R^\gamma)} \left( \int_{B(z,R)} |k(x, y) - k(x,z)|R^{-Q_0/p}\diff y
        \right)^p\diff x.
    \end{align*}
    Using H\"older's inequality with exponent $1/p$ and the kernel estimates from Theorem \ref{theo:kernel-estimate-dyadic}, we have that 
    \begin{align*}
        I_2  & \le\sum_j \left(\int_{B(z,R)} \int_{A_j(z,R^\gamma)}   |k(x, y) - k(x,z)|\diff x\diff y
        \right)^p \\
        & \quad \times\left(\int_{A_j(z,R^\gamma)} R^{-Q/(1-p)}\diff x
        \right)^{1-p} \\
        & \le C\|\sigma\|_{(l)}\sum_j 2^{-2jp/\kappa}R^{p(1-2\gamma/\kappa)}|B(z,R)|^p R^{-Q} |A_j(z,R^\gamma)|^{1-p} \\
        & \le C\|\sigma\|_{(l)}\sum_j2^{-2jp/\kappa}R^{p(1-2\gamma/\kappa)} R^{Qp}R^{-Q}2^{jn(1-p)}R^{n\gamma(1-p)} \\
        & \le C\|\sigma\|_{(l)} R^{-\gamma[p(n+2/\kappa)-n]+p(Q+1)-Q} \sum_j2^{-j [p(n+2/\kappa)-n] }.
    \end{align*}
    Since we require that $p(n+2/\kappa)-n>0, $ we have that 
    \begin{equation*}
        I_2 \le C\|\sigma\|_{(l)}R^{-\gamma[p(n+2/\kappa)-n]+p(Q+1)-Q},
    \end{equation*}
    and we must have that 
    \begin{equation}\label{eq:gamma-upper}
        \gamma \le \frac{p(Q+1)-Q}{p(n+2/\kappa)-n},
    \end{equation}
    which is an increasing function of $p$ and it reaches 
    \begin{equation*}
        \gamma \le \frac{\kappa}{2}
    \end{equation*}
    when $p=1.$ Thus, we must equal the right sides of (\ref{eq:gamma-lower}) and (\ref{eq:gamma-upper}) to optain the critical value of $p_0$ 
    \begin{equation*}
       Q\frac{2-p_0(1+2\beta/n)}{n(2-p_0)} =\frac{p_0(Q+1)-Q}{p_0(n+2/\kappa)-n},
    \end{equation*}
    which results in 
    \begin{equation*}
        \frac{1}{p_0}=\frac{1}{2} + \frac{\beta Q(1/n+\kappa/4) }{Q+\beta  Q\kappa /2 - n\kappa/2  }.
    \end{equation*}
    Now we follow with the case $R\ge1.$ To estimate $I_1$ we use H\"older's inequality with exponent $2/p$ and the $L^2$ continuity to obtain that 
    \begin{align*}
        I_1 & \le  \left(\int_{B_E(z,2R)} |\sigma(x, D)b(x)|^2\diff x
        \right)^{p/2} \left(\int_{B_E(z,2R)}\diff x
        \right)^{(2-p)/2} \\
        & = \|\sigma(x, D)b\|_{L^2}^p |B_E(z, R)|^{(2-p)/2} \\
        &\le C \|\sigma\|_{(l)}\|b\|_{L^2}^pR^{n(2-p)/2}\\
        &\le C \|\sigma\|_{(l)}\left(\int_{B(z,R)}|B(z,R)|^{-2/p}\diff x
        \right)^{p/2}R^{n(2-p)/2} \\
        & \le C \|\sigma\|_{(l)} R^{-Q}R^{Qp/2} R^{n(2-p)/2} \\
        & = C \|\sigma\|_{(l)}R^{(n-Q)(2-p)/2} \le C \|\sigma\|_{(l)},
    \end{align*}
    since $Q\ge n.$ For $I_2$, we use the cancellation property of $b(x)$ and the kernel estimates from Theorem \ref{theo:kernel-estimate-dyadic} as before to obtain that
    \begin{align*}
        I_2  & \le\sum_j \left(\int_{B(z,R)} \int_{A_j(z,R)}   |k(x, y) - k(x,z)|\diff x\diff y
        \right)^p \\
        & \quad \times\left(\int_{A_j(z,R)} R^{-Q/(1-p)}|\diff x|
        \right)^{1-p} \\
        & \le C\|\sigma\|_{(l)}\sum_j 2^{-jp}|B(z,R)|^p R^{-Q} |A_j(z,R)|^{1-p} \\
        & \le C\|\sigma\|_{(l)}\sum_j2^{-jp} R^{Qp}R^{-Q}2^{jn(1-p)}R^{n(1-p)} \\
        & \le C\|\sigma\|_{(l)} R^{(n-Q)(1-p)} \sum_j2^{-j [p-n(1-p)] }.
    \end{align*}
    This last term is bounded by a constant since $Q\ge n$ and $p(n+1)-n > 0.$ If $\beta\ge n/2,$ we may choose $\varepsilon>0$ so that $\beta -\varepsilon<n/2.$ Then, we apply the previous case and take $\varepsilon\to \beta-n/2.$ This completes the proof.
\end{proof}
        
\begin{remark}
    Notice that when $Q=n,$ we obtain the result obtained in \'Alvarez and Hounie \cite{alvarez-hounie} with $\rho=\kappa/2$ given by 
    \begin{equation*}
        \frac{1}{p_0} = \frac{1}{2} + \frac{\beta(2/\kappa+n/2)}{n(2/\kappa-1+\beta)}.
    \end{equation*}
\end{remark}

\bibliographystyle{amsplain}

\end{document}